\documentclass[a4paper,11pt]{amsart}
\usepackage{amsmath,amsfonts,amssymb,amsthm,latexsym,amscd}
\usepackage[dvips]{graphicx}
\usepackage[T1]{fontenc}
\usepackage{parskip}

\def\A{\mathcal{A}}
\def\a{\alpha}
\def\AA{A^2}

\def\AG{\langle A,G\rangle}

\def\D{\mathcal{D}}
\def\DD{D^2}
\def\eps{\varepsilon}
\def\Hom{\operatorname{Hom}}
\def\Im{\operatorname{Im}}
\def\lk{\operatorname{lk}}

\def\nA{\nabla_{asc}}
\def\nD{\nabla_{des}}
\def\nAD{\nabla_{asc-des}}

\def\S{\Sigma}

\def\s{\operatorname{sign}}
\def\sm{\smallsetminus}
\def\tG{\widetilde{G}}
\def\Z{\mathbb{Z}}
\newcommand\ris[5]{\raisebox{#1mm}{\hspace{#2mm}\includegraphics[width=#3mm]{#4.eps}\hspace{#5mm}}}
\newcommand{\rb}{\raisebox}
\newcommand{\ig}{\includegraphics}
\newcommand\risS[6]{\rb{#1pt}[#5pt][#6pt]{\begin{picture}(#4,15)(0,0)
  \put(0,0){\ig[width=#4pt]{#2.eps}} #3
     \end{picture}}}

\def\kar#1#2#3#4#5#6{\begin{array}{|c|}\hline
\risS{-17}{o3-#1}{\put(12,36){$\scriptstyle #3$}
\put(4,2){$\scriptstyle #4$}\put(37,9){$\scriptstyle #5$}}{39}{27}{20}\\
\hline \risS{-17}{o3-#2}{\put(12,36){$\scriptstyle #3$}
\put(4,2){$\scriptstyle #4$}\put(37,9){$\scriptstyle #5$}}{39}{27}{20} \\
\hline \end{array}}

\newtheorem{thm}{Theorem}[section]
\newtheorem{lem}[thm]{Lemma}
\newtheorem{prop}[thm]{Proposition}
\newtheorem{cor}[thm]{Corollary}
\newtheorem{defn}[thm]{Definition}

\newtheorem{rem}[thm]{Remark}
\newtheorem{ex}[thm]{Example}

\begin{document}

\title{Link invariants via counting surfaces}

\author{Michael Brandenbursky}

\vspace{2cm}

\begin{abstract}
A Gauss diagram is a simple, combinatorial way to present a
knot. It is known that any Vassiliev invariant may be obtained
from a Gauss diagram formula that involves counting (with signs
and multiplicities) subdiagrams of certain combinatorial types.
These formulas generalize the calculation of a linking number by
counting signs of crossings in a link diagram.
Until recently, explicit formulas of this type were known only
for few invariants of low degrees. In this paper we present
simple formulas for an infinite family of invariants in terms
of counting surfaces of a certain genus and number of boundary components in a Gauss diagram.
We then identify the resulting invariants with certain derivatives
of the HOMFLYPT polynomial.
\end{abstract}

\maketitle

\section{Introduction.}

In this paper we consider link invariants arising from the Conway
and HOMFLYPT polynomials. The HOMFLYPT polynomial $P(L)$
is an invariant of an oriented link $L$ (see e.g. \cite{FYHLMO},
\cite{LM}, \cite{PT}). It is a Laurent polynomial in two
variables $a$ and $z$, which satisfies the following skein relation:

\begin{equation}\label{eq:Homfly-skein}
aP\left(\ris{-4}{-1}{10}{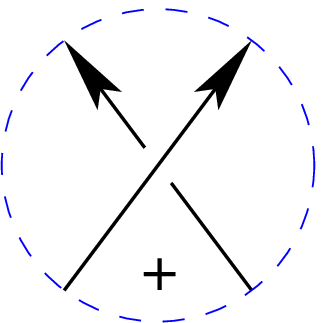}{-1.1}\right)-
a^{-1}P\left(\ris{-4}{-1}{10}{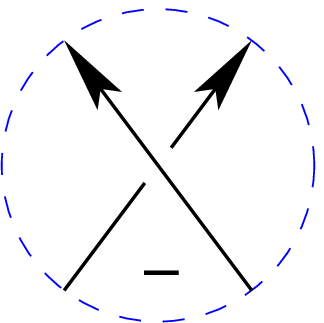}{-1.1}\right)=
zP\left(\ris{-4}{-1.1}{10}{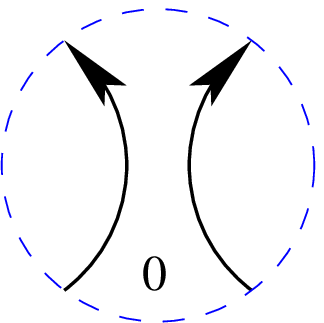}{-1.1}\right).
\end{equation}

The HOMFLYPT polynomial is normalized in the following way.
If $O_m$ is an $m$-component unlink, then
$P(O_m)=\left(\frac{a-a^{-1}}{z}\right)^{m-1}$. The Conway
polynomial $\nabla$ may be defined as $\nabla(L):=P(L)|_{a=1}$.
This polynomial is a renormalized version of the Alexander
polynomial (see e.g. \cite{Conway}, \cite{Likorish}). All coefficients
of $\nabla$ are finite type or Vassiliev invariants.

One of the mainstream and simplest techniques for producing
Vassiliev invariants are so-called Gauss diagram formulas (see
\cite{GPV}, \cite{PV}). These formulas generalize the calculation
of a linking number by counting subdiagrams of special
geometric-combinatorial types with signs and weights in a given
link diagram. This technique is also very helpful in the rapidly
developing field of virtual knot theory (see \cite{Ka}), as well
as in $3$-manifold theory (see \cite{MP}).

Until recently, explicit formulas of this type were known only for
few invariants of low degrees. The situation has changed with
works of Chmutov-Khoury-Rossi \cite{CKR} and Chmutov-Polyak
\cite{CP}, see also \cite{KrP} for the case of string links. In \cite{CKR} Chmutov-Kho\-ury-Rossi presented an
infinite family of Gauss diagram formulas for all coefficients of
$\nabla(L)$, where $L$ is a knot or a two-component
link. We explain how each formula for the coefficient
$c_n$ of $z^n$ is related to a certain count of orientable
surfaces of a certain genus, and with one boundary component. The
genus depends only on $n$ and the number of the components of $L$.
These formulas may be viewed as a certain combinatorial analog of
Gromov-Witten invariants.

In this work we generalize the result of Chmutov-Khoury-Rossi
to links with arbitrary number of components. We present a direct
proof of this result, without any prior assumption on the existence
of the Conway polynomial. It enables us to present two different
extensions of the Conway polynomial to long virtual links. We
compare these extensions with the existing versions of the
Alexander and Conway polynomials for virtual links, and show
that they are new. In particular, we give a new proof of the fact
that the famous Kishino knot $K_T$ \cite{Ki} is non-classical, by
calculating these polynomials for $K_T$. Later we show that these formulas
may be modified by counting only certain, so-called irreducible,
subdiagrams.

This leads to a natural question: how to produce link invariants
by counting orientable surfaces with an arbitrary number of
boundary components? In this paper we deal with a model case,
when the number of boundary components is two. We modify
Chmutov-Khoury-Rossi construction and present an infinite
family of Gauss diagram formulas for the coefficients of the first
partial derivative of the HOMFLYPT polynomial, w.r.t. the
variable $a$, evaluated at $a=1$. This family is related, in a
similar way, to the family of orientable surfaces with two
boundary components. Later we modify these formulas in case of knots.

In the forthcoming paper \cite{B} we show that, in case of closed braids, a similar count of
orientable surfaces with $n$ boundary components (up to some
normalization) is related to an infinite family of Gauss diagram
formulas for the coefficients of the $n-1$ derivative, w.r.t.
the variable $a$ in the HOMFLYPT polynomial, evaluated at $a=1$.

\textbf{Acknowledgments.} I would like to thank Michael Polyak, who has introduced me to this subject, guided and helped me a
lot while I was working on this paper. I also would like to thank the referee for careful reading of this paper and for his/her useful comments and remarks.

Part of this work has been done during the author's stay at Max Planck Institute for Mathematics in Bonn. The author wishes to express his gratitude to the Institute for the support and excellent working conditions.


\section{Gauss diagrams and arrow diagrams}
\label{sec-Conway}
In this section we recall a notion of Gauss diagrams, arrow diagrams and Gauss diagram formulas. We then define a special type of arrow diagrams which will be used to define Gauss diagram formulas for coefficients of the Conway polynomial, and for coefficients of some other polynomials derived from the HOMFLYPT polynomial.

\subsection{Gauss diagrams of classical and virtual links}
Gauss diagrams (see e.g. \cite{Fiedler}, \cite{GPV}, \cite{PV}) provide a simple combinatorial way to encode classical and virtual links.

\begin{defn}\rm
Given a classical (possibly framed) link diagram $D$, consider a collection of oriented circles parameterizing it. Unite two preimages of every crossing of $D$ in a pair and connect them by an arrow, pointing from the overpassing preimage to the underpassing one. To each arrow we assign a sign (writhe) of the corresponding crossing. The result is called the \textit{Gauss diagram} $G$ corresponding to $D$.
\end{defn}

We consider Gauss diagrams up to an orientation-preserving diffeomorphisms of the circles.
In figures we will always draw circles of the Gauss diagram with a counter-clockwise orientation.

\begin{ex}\rm
Diagrams of the trefoil knot and the Hopf link, together with the corresponding Gauss diagrams, are shown in the following picture.
\begin{equation*}
\quad\ris{-8}{-3}{40}{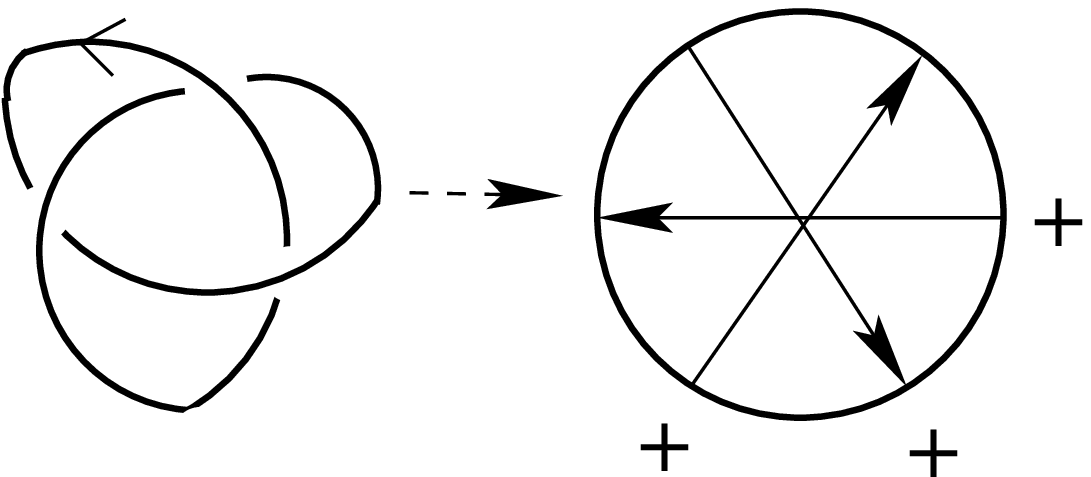}{-1.1}\quad\quad\quad\ris{-6.5}{-3}{74}{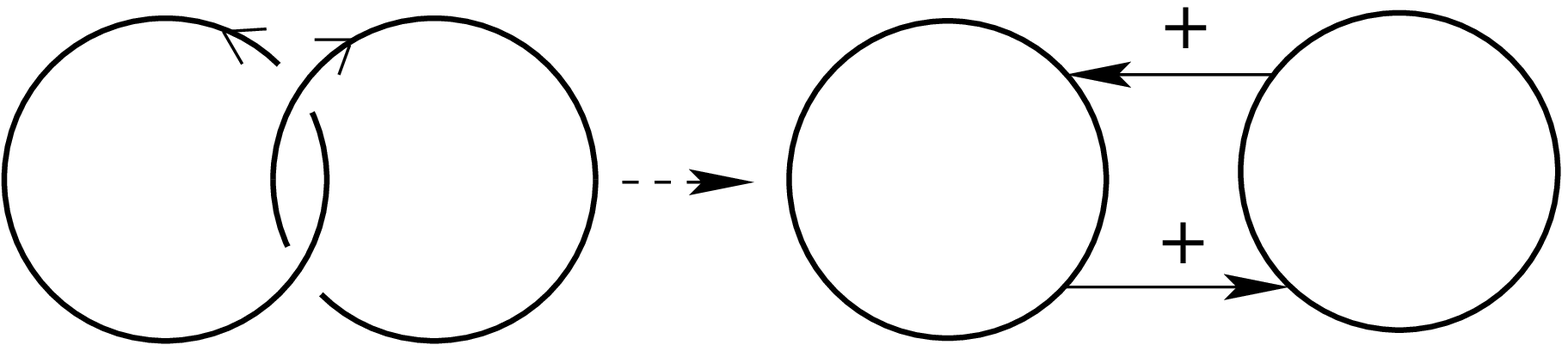}{-1.1}
\end{equation*}
\end{ex}

A classical link can be uniquely reconstructed from the corresponding Gauss diagram \cite{GPV}.  Many fundamental knot invariants, such as the knot group and the Alexander polynomial, may be easily obtained from the Gauss diagram. We are going to work with \textit{based Gauss diagrams}, i.e. Gauss diagrams with a base point (different from the endpoints of the arrows) on one of the circles. If we cut a based circle at the base point, we will get a Gauss diagram of a long link, see Figure \ref{fig:Hopf+long+Gauss}.

\begin{figure}[htb]
\centerline{\includegraphics[height=0.93in]{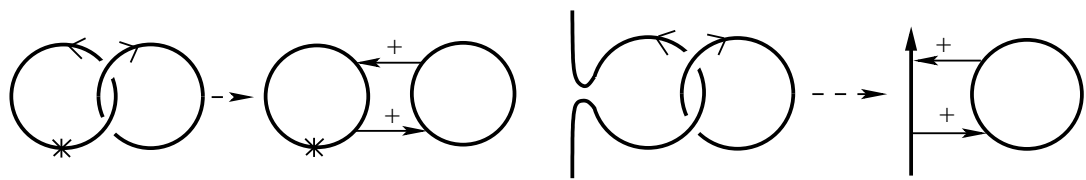}}
\caption{\label{fig:Hopf+long+Gauss} Diagrams of based and long classical Hopf links together with the associated  Gauss diagrams.}
\end{figure}

\begin{figure}[htb]
\centerline{\includegraphics[height=0.8in]{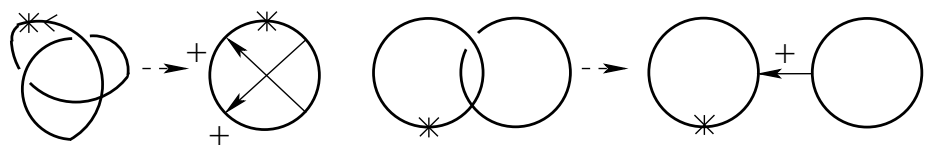}}
\caption{\label{fig:trefoil+hopf+virtual} Virtual trefoil and a virtual Hopf link with the corresponding Gauss diagrams.}
\end{figure}

Note that not every collection of circles with signed arrows is realizable as a Gauss diagram of a classical link, see Figure \ref{fig:trefoil+hopf+virtual}. The Gauss diagram of a virtual link diagram is constructed in
the same way as for a classical link diagram, but all virtual crossings are disregarded, see Figure \ref{fig:trefoil+hopf+virtual}. Similarly to the case of long classical links, each non-realizable Gauss diagram with a base point represents a long  virtual link.

Two Gauss diagrams represent isotopic classical/virtual links (long links) if and only if they are related by a finite number of Reidemeister moves for Gauss diagrams (applied away from the base point) shown in Figure \ref{fig:Reidem}, where $\eps=\pm1$, see e.g. \cite{CDBook, Oestlund, P}. Note that not all Reidemeister moves are shown in Figure \ref{fig:Reidem} (for example third Reidemeister moves with at least one negative crossing are not shown), but their generating set is, see \cite{P}.

\begin{figure}[htb]
\begin{equation*}
\Omega_{1}:\quad\ris{-7}{-3}{33}{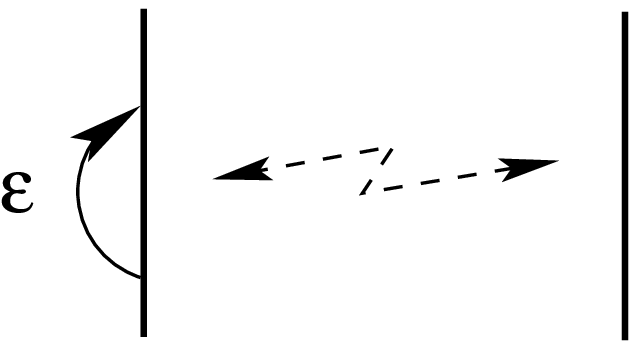}{-1.1}\quad\quad\quad\Omega_{2}:\quad\ris{-7}{-3}{49}{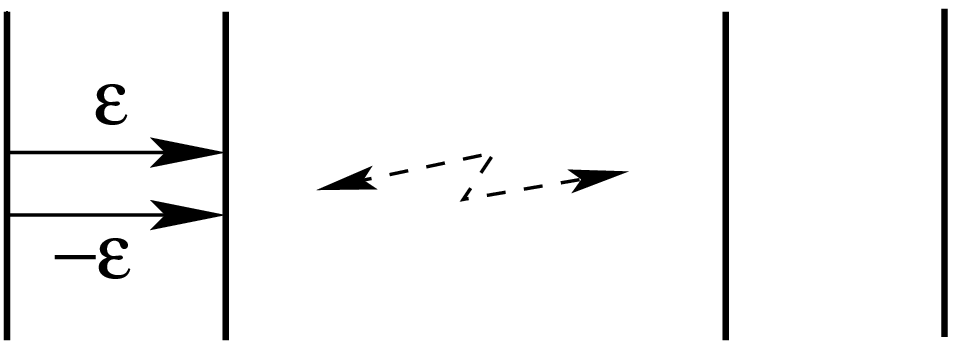}{-1.1}
\end{equation*}
\begin{equation*}
\Omega_{3}:\quad\ris{-8}{-3}{70}{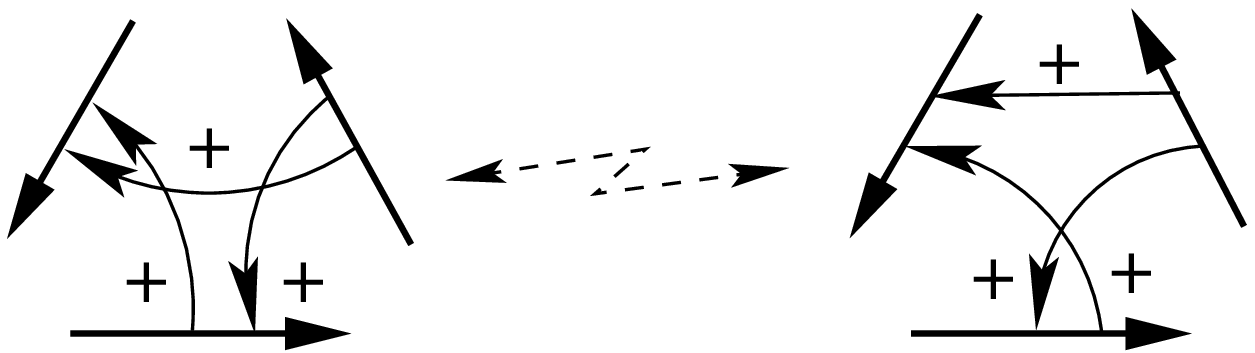}{-1.1}
\end{equation*}
\caption{\label{fig:Reidem} Reidemeister moves of Gauss diagrams.}
\end{figure}

Two Gauss diagrams represent isotopic classical framed links if and only if they are related by a finite number of Reidemeister moves for framed Gauss diagrams. It suffices to consider $\Omega_2$ and $\Omega_3$ of Figure \ref{fig:Reidem} and substitute the move $\Omega_1$ by the move

\begin{equation*}
\Omega_1^F:\quad\ris{-7}{-3}{60}{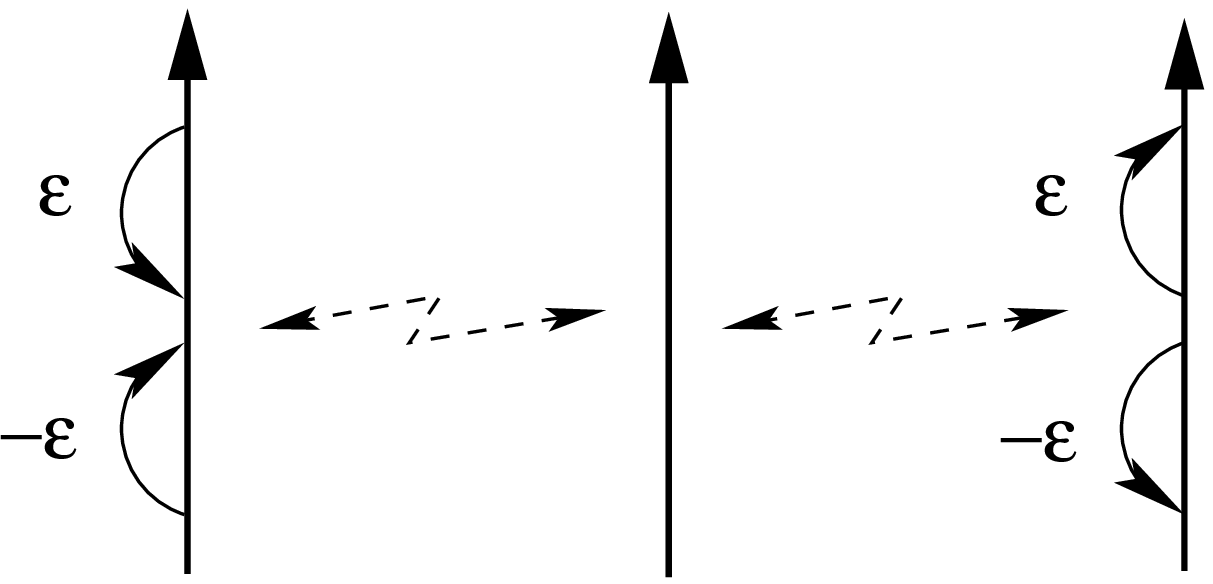}{-1.1}
\end{equation*}

Note that segments involved in $\Omega_2$ or $\Omega_3$ may lie on different components of the link and the order in which they are traced along the link may be arbitrary.

\subsection{Arrow diagrams and Gauss diagram formulas}

An \textit{arrow diagram} is a modification of a notion of a Gauss
diagram, in which we forget about realizability and
signs of arrows, see Figure \ref{fig:non-realizable}.

\begin{figure}[htb]
\centerline{\includegraphics[height=0.6in]{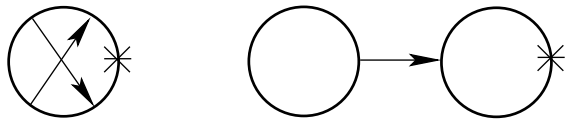}}
\caption{\label{fig:non-realizable}Connected arrow diagrams.}
\end{figure}

In other words, an arrow diagram consists of a number of oriented circles with several arrows connecting pairs of distinct points on them. We consider these diagrams up to
orientation-preserving diffeomorphisms of the circles.
An arrow diagram is \textit{based}, if a base point (different from the
end points of the arrows) is marked on one of the circles. An arrow
diagram is connected, if it is connected as a graph. Further we will
consider only based connected arrow diagrams, so we will omit mentioning
these requirements throughout this chapter, unless a misunderstanding is
likely to occur. In figures we will always draw the circles of an arrow
diagram with a counter-clockwise orientation.

M. Polyak and O. Viro suggested \cite{PV} the following approach to
compute link invariants using Gauss diagrams.

\begin{defn}\rm
Let $A$ be an arrow diagram with $m$ circles and let $G$ be a based
Gauss diagram of an $m$-component oriented (long, virtual) link. A \textit{homomorphism}
$\phi:A\rightarrow G$ is an orientation preserving homeomorphism
between each circle of $A$ and each circle of $G$, which maps a base
point of $A$ to the base point of $G$ and induces an injective map
of arrows of $A$ to the arrows of $G$. The set of arrows in
$\Im(\phi)$ is called a \textit{state} of $G$ induced by $\phi$ and is
denoted by $S(\phi)$. The \textit{sign} of $\phi$ is defined as
$\s(\phi)=\prod_{\a\in S(\phi)}sign(\a)$. A set of all homomorphisms
$\phi:A\to G$ is denoted by $\Hom(A,G)$.
\end{defn}

Note that since the circles of $A$ are mapped to circles of $G$,
a state $S$ of $G$ determines both the arrow diagram $A$ and the
map $\phi:A\to G$ with $S=S(\phi)$.

\begin{defn}
\rm A \textit{pairing} between an  arrow diagram $A$ and $G$ is defined by
$$\AG=\sum_{\phi\in\Hom(A,G)}\s(\phi).$$
\end{defn}

For an arbitrary arrow diagram $A$ the pairing $\AG$ does not represent a link invariant, i.e. it depends on the choice of a Gauss diagram of a link. However, for some special linear combinations of arrow diagrams the result is independent of the choice of $G$, i.e. does not change under the Reidemeister moves for Gauss diagrams. Using a slightly modified definition of arrow diagrams Goussarov, Polyak and Viro showed in \cite{GPV} that each real-valued Vassiliev invariant of long knots may be obtained this way. In particular, all coefficients of the Conway polynomial may be obtained using suitable combinations of arrow diagrams.

\subsection{Surfaces corresponding to arrow diagrams}

Given an arrow diagram $A$, we define an oriented surface $\S(A)$ as
follows. Firstly, replace each circle of $A$ with an oriented disk
bounding this circle. Secondly, glue $1$-handles to boundaries of
these disks using each arrow as a core of a ribbon. See Figure
\ref{fig:surface}.

\begin{figure}[htb]
\centerline{\includegraphics[height=0.7in]{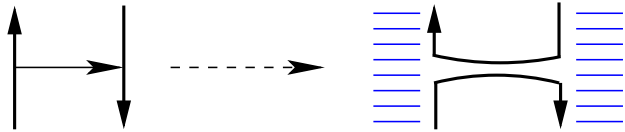}}
\caption{\label{fig:surface} Constructing a surface from an arrow diagram.}
\end{figure}

\begin{defn}\rm
By the \textit{genus} and the \textit{number of boundary components}
of an arrow diagram $A$ we mean the genus and the number of boundary
components of $\S(A)$.
\end{defn}

\begin{rem}\rm
Let $A$ be an arrow diagram with $n$ arrows and $m$ circles. Then the Euler characteristic $\chi$ of $\S(A)$ equals to $\chi(\S(A))=m-n$. If $A$ is connected,
$n\geq m-1$. If $A$ has one boundary component, $n\neq m (\rm{mod}2)$.
\end{rem}

\begin{ex}\rm
The arrow diagram with one circle in Figure \ref{fig:non-realizable} is of genus
one, while the other arrow diagram in the same figure is
of genus zero. Both of them have one boundary component.
\end{ex}

Further we will work only with based connected arrow diagrams with one or two
boundary components.

\subsection{Ascending and descending arrow diagrams}
In this subsection we define a special type of arrow diagrams with one and two boundary components.

\begin{defn}\rm
Let $A$ be a based arrow diagram with one boundary component. As we go
along the boundary of $\S(A)$ starting from the base point, we pass
on the boundary of each ribbon twice: once in the direction of its
core arrow, and once in the opposite direction. $A$ is \textit{ascending}
(respectively, \textit{descending}) if we pass each ribbon of $\S(A)$ first
time in the direction opposite to its core arrow (respectively, in the
direction of its core arrow).
\end{defn}

\begin{rem} \rm
In order to define the notion of ascending and descending arrow diagrams we used the fact that all arrow diagrams are based and connected. The position of the base point in a connected arrow diagram is essential to define an order of passage.
\end{rem}

\begin{ex}\rm
Arrow diagrams presented below are
ascending (a), descending (b) and neither ascending nor descending (c).
\end{ex}

\begin{figure}[htb]
\centerline{\includegraphics[height=0.9in]{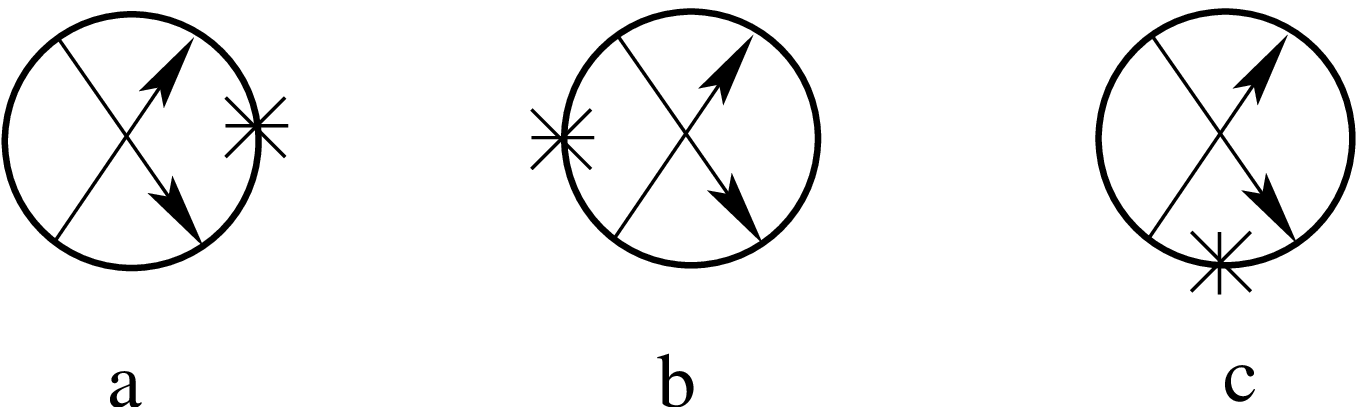}}
\label{fig:arrow_diagrams}
\end{figure}

Denote by $\mathcal{A}_{n,m}$ (respectively, $\mathcal{D}_{n,m}$) the set of all
ascending (respectively, descending) arrow diagrams with $n$ arrows, $m$ circles and one boundary component.
\begin{ex}\rm
The sets $\A_{2,1}$ and $\D_{2,1}$  are presented below.
$$\A_{2,1}:=\quad\ris{-4}{-3}{9}{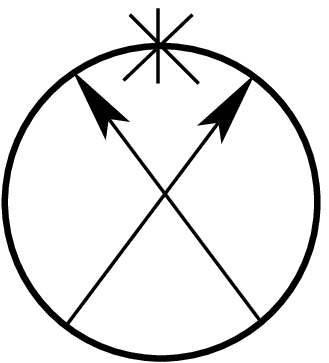}{-1.1}\quad\textrm{and}\quad \D_{2,1}:=\quad\ris{-4}{-3}{9}{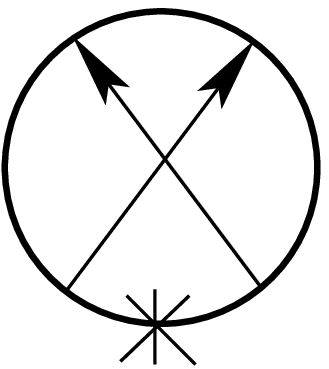}{-1.1}$$
\end{ex}

\begin{defn} \rm
Let $G$ be any Gauss diagram with $m$ circles. We set
$$A_{n,m}(G):=\sum_{A\in \mathcal{A}_{n,m}}\AG\qquad D_{n,m}(G):=\sum_{A\in \mathcal{D}_{n,m}}\AG$$
and define the following polynomials:
$$\nA(G):=\sum_{n=0}^\infty A_{n,m}(G) z^n \qquad \nD(G):=\sum_{n=0}^\infty D_{n,m}(G) z^n$$
\end{defn}

These polynomials will play an important role in the next section. Now we generalize a notion of ascending (descending) arrow diagram to arrow diagrams with two boundary components. We would like to point out that in \cite{B} this notion is generalized for arrow diagrams with arbitrary number of boundary components.

\begin{defn}\rm
Let $A$ be an arrow diagram with two boundary components. As we go
along the component of $\partial\S(A)$ starting from the base point, we pass
on the boundary of each ribbon once or twice (since $A$ is connected we must pass all ribbons at least once). We call core arrows, which we pass only in one direction, the \textit{separating arrows}. Now we place another starting point \textrm{$\bullet$} on the second component of $\partial\S(A)$ near the first separating arrow which we encounter in the passage, and start going along this component of $\partial\S(A)$. $A$ is \textit{ascending} (respectively, \textit{descending}) if we pass each ribbon of $\S(A)$ first time in the direction opposite to its core arrow (respectively, in the direction).
\end{defn}

\begin{ex} \rm
Arrow diagrams below have two boundary components. Diagram (a) is ascending and diagram (b) is descending. Separating arrows are shown in bold.
\begin{equation*}
\ris{-7}{-1}{38}{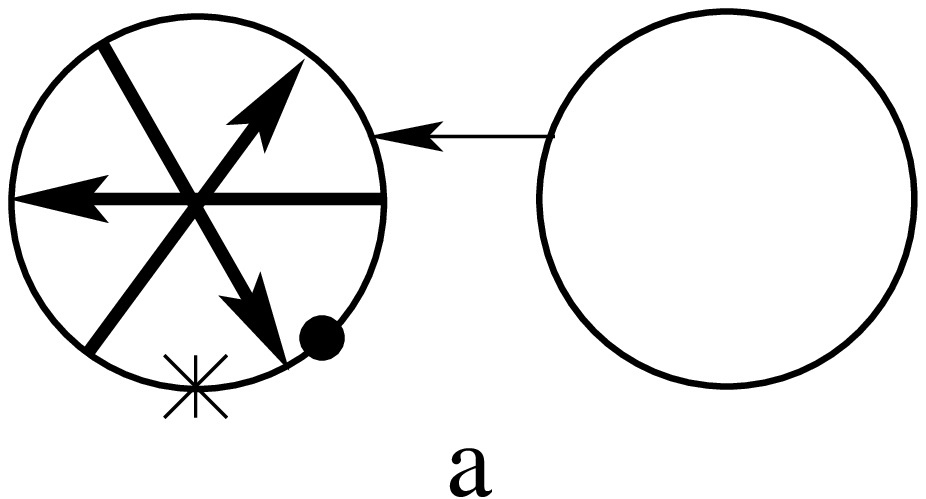}{-1.1}
\hspace{1.5cm}
\ris{-7}{-1}{61}{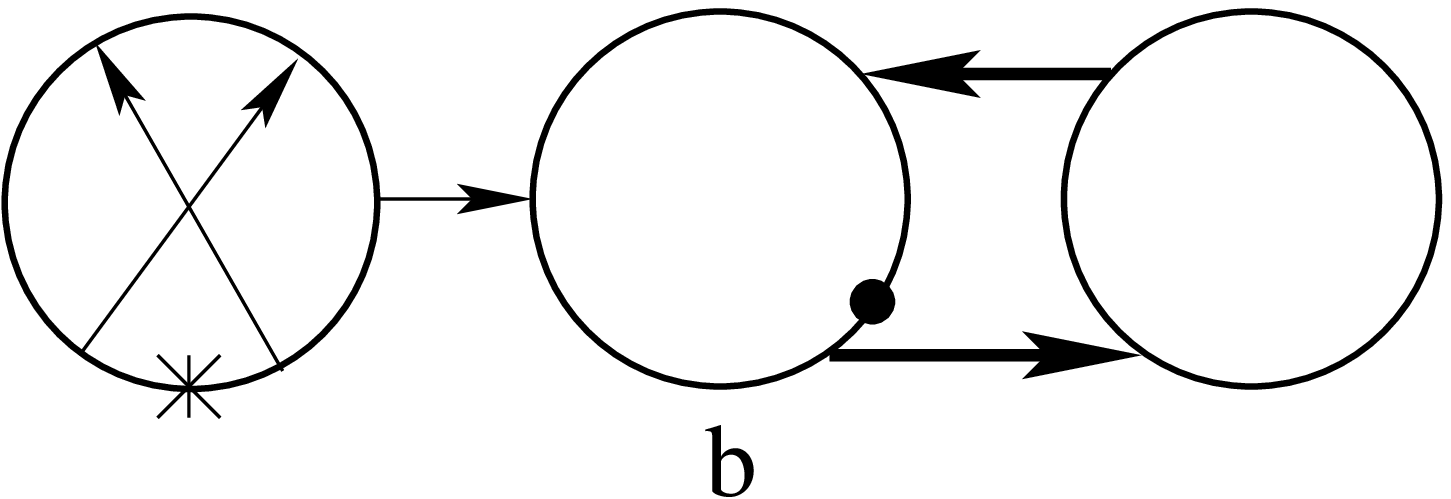}{-1.1}
\end{equation*}
\end{ex}

Denote by $\A_{n,m}^2$ (respectively, $\D_{n,m}^2$) the set of all
ascending (respectively, descending) arrow diagrams with $n$ arrows,
$m$ circles and two boundary components.

\begin{ex}\rm
All diagrams in the set $\A^2_{2,2}$ are presented below.
$$
\ris{-3}{-3}{24}{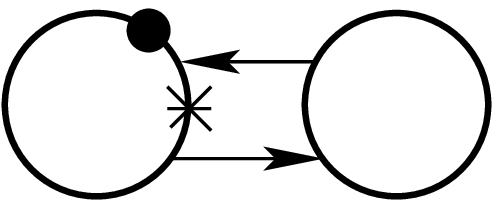}{-1.1}
\hspace{0.9cm}\ris{-4}{-3}{24}{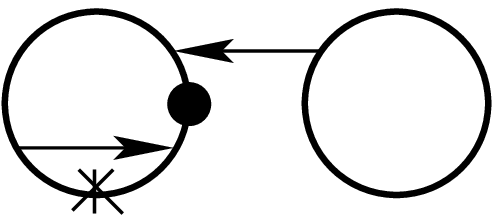}{-1.1}
\hspace{0.9cm}\ris{-4}{-3}{24}{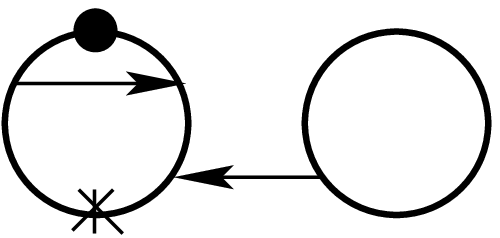}{-1.1}
\hspace{0.9cm}\ris{-4}{-3}{24}{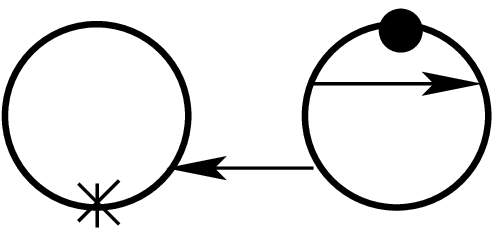}{-1.1}
\hspace{0.9cm}\ris{-3}{-3}{24}{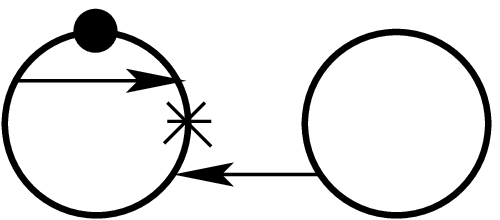}{-1.1}
$$
\end{ex}

Let $G$ be any Gauss diagram with $m$ circles. We set
$$\AA_{n,m}(G):=\sum_{A\in \mathcal{A}^2_{n,m}}\AG\qquad \DD_{n,m}(G):=\sum_{A\in \mathcal{D}^2_{n,m}}\AG.$$

A state $S(\phi)$ corresponding to $\phi:A\to G$ for an ascending
(respectively descending) diagram $A$ with one or two boundary
components will be also called \textit{ascending} (respectively
\textit{descending}). It is useful to reformulate this notion in terms
of a tracing of a diagram $G$.

\begin{defn}\rm
Given $\phi:A\to G$, a passage along the boundary of the surface
$\S(A)$ induces a \textit{tracing} of $G$: we follow an arc of a circle
of $G$ starting from the base point until we hit an arrow in $S(\phi)$,
turn to this arrow, then continue on another arc of $G$ following
the orientation and so on, until we return to the base point. In case of
two boundary components we repeat the same procedure starting near
the image of the first separating arrow. Then a state $S(\phi)$ is
\textit{ascending} (respectively \textit{descending}), if we approach
every arrow in the tracing first time at its head (respectively at its tail).
\end{defn}

\subsection{Separating states}
\label{subsec-Enhanced-states}
In this subsection we define a notion of a separating state. This notion will be extensively used in the following sections.
\begin{figure}[htb]
\centerline{\includegraphics[height=1in]{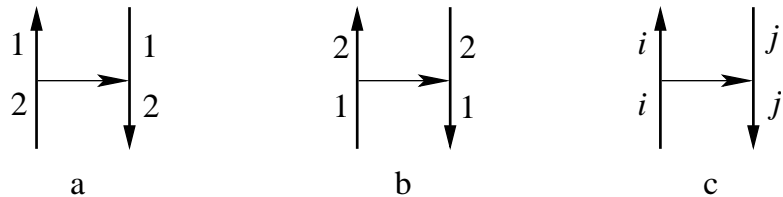}}
\caption{\label{fig:asc-des-arcs} Ascending and descending labeling. Here $i,j\in\{1,2\}$.}
\end{figure}

\begin{defn}\rm
Let $G$ be a based Gauss diagram.
An \textit{ascending (respectively descending) separating state $S$ of $G$} is a state $S$ of $G$, together with a labeling of arcs of $G$ (i.e., intervals of circles of $G$ between endpoints of arrows) by $1$ and $2$ such that:
\begin{enumerate}
\item
Each arc near $\a\in S$ is labeled as in Figure \ref{fig:asc-des-arcs}a (respectively in Figure \ref{fig:asc-des-arcs}b).
\item
Each arc near $\a\notin S$ is labeled as in Figure \ref{fig:asc-des-arcs}c.
\item
An arc with a base point is labeled by $1$.
\end{enumerate}
\end{defn}

Every separating (ascending or descending) state  $S$ in $G$ defines a new Gauss diagram $G_S$ with labeled circles as follows:
We smooth each arrow in $G$ which belongs to $S$, see Figure \ref{fig:smoothing1}, and denote the resulting smoothed Gauss diagram by $G_S$. Each circle in $G_S$ is labeled by $i$, if it contains an arc labeled by $i$.

\begin{figure}[htb]
\centerline{\includegraphics[height=0.75in]{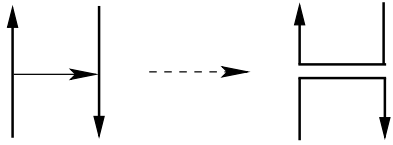}}
\caption{\label{fig:smoothing1} Smoothing of an arrow.}
\end{figure}

Now we return to arrow diagrams with two boundary components. Let $A\in~\mathcal{A}_{n,m}^2$ or $A\in\mathcal{D}_{n,m}^2$. We denote by $\sigma(A)$ the set of separating arrows in $A$ and label the arcs of circles in $A$ by $1$ if the corresponding arc belongs to the first boundary component of $\S(A)$ and by $2$ otherwise.

Note that each homomorphism $\phi:A\rightarrow G$ induces an ascending or descending separating state $S$ of $G$, by taking $S=\phi(\sigma(A))$ and labeling each arc of $G$ by the same label as the corresponding arc of $A$.

\begin{defn} \rm
Let $G$ be a based Gauss diagram with $m$ circles. Let $S$ be an ascending (respectively descending) separating state of $G$, $A\in\mathcal{A}_{n,m}^2$ (respectively $A\in~\mathcal{D}_{n,m}^2$), and $\phi:A\rightarrow G$. We say that $\phi$ is \textit{$S$-admissible}, if an ascending (respectively descending) separating state induced by $\phi$ coincides with $S$.
\end{defn}

\begin{defn} \label{D:G_S}\rm
Let $S$ be an ascending (respectively descending) separating state of $G$, and $A\in\mathcal{A}_{n,m}^2$ (respectively $A\in\mathcal{D}_{n,m}^2$). In each case we define an \textit{$S$-pairing} $\AG_S$ by:
$$\AG_S:=\sum\limits_{\phi:A\rightarrow G}\s(\phi),$$
where the summation is over all $S$-admissible $\phi:A\rightarrow G$. We set
$$A^2_{n,m}(G)_S:=\sum_{A\in\A_{n,m}^2}\AG_S\quad
\textrm{and}\quad  D^2_{n,m}(G)_S:=\sum_{A\in\D_{n,m}^2}\AG_S.$$
\end{defn}

\section{Counting surfaces with one boundary component}
\label{sec-1-comp}

In this section we review Gauss diagram formulas for coefficients of the Conway polynomial $\nabla$ obtained in \cite[Theorem 3.5]{CKR} for classical knots and $2$-component classical links.

\begin{thm}[\cite{CKR}]
Let $G$ be a Gauss diagram of a classical knot or $2$-component classical link $L$.
\begin{equation}
\nA(G)=\nD(G)=\nabla(L),
\end{equation}
where $\nabla(L)$ is the Conway polynomial.
\end{thm}

Let $G$ be a Gauss diagram with $m$ circles. We give a direct proof
of the invariance of both $\nA(G)$ and $\nD(G)$ under the Reidemeister moves which do not involve a base point. This allows us to extend this result to $m$-component classical links and also to define two different generalizations of the Conway polynomial to long virtual links. At the end of this section we present some properties of these polynomials.

\subsection{Invariants of long links}
In this subsection we generalize the result of \cite{CKR} to $m$-component (classical or virtual) long links.

\begin{thm}\label{thm:Reidemeister-basepoint}
Let $G$ be a Gauss diagram of an $m$-component (classical or virtual) long link $L$.
Then $\nA(G)$ and $\nD(G)$ define polynomial invariants of long links, i.e. do not depend on the choice of $G$.
\end{thm}

\begin{proof}
We prove that $\nA(G)$ is an invariant of an underlying link. The proof for $\nD(G)$ is the same.
It suffices to show that $ A_{n,m}(G)$ is invariant under the Reidemeister moves $\Omega_1$, $\Omega_2$ and $\Omega_3$ of Figure \ref{fig:Reidem} applied away from the base point.
Let $G$ and $\tG$ be two Gauss diagrams that differ by an application of $\Omega_{1}$, so that $\tG$ has one
additional isolated arrow $\a$ on one of the circles.
Ascending states of $G$ are in bijective correspondence with
ascending states of $\tG$ which do not contain $\a$. But $\a$
cannot not be in the image of $\phi:A\to \tG$ with $A\in \mathcal{A}_{n,m}$, because $A$ should have one boundary component. Thus $A_{n,m}(G)=A_{n,m}(\tG)$.

Let $G$ and $\tG$ be two Gauss diagrams that differ by an application
of $\Omega_2$, so that $\tG$ has two
additional arrows $\a_{\eps}$ and $\a_{-\eps}$, see Figure \ref{fig:Reidem}.
Ascending states of $G$ are in bijective correspondence with
ascending states of $\tG$ which do not contain $\a_{\pm\eps}$. Note that both $\a_{\eps}$ and $\a_{-\eps}$ can not be in the image of $\phi:A\to \tG$ with $A\in \mathcal{A}_{n,m}$ because $A$ has one boundary
component. Ascending states of $\tG$ which contain one of $\a_{\pm\eps}$ come in pairs $S\cup\a_{\eps}$ and $S\cup\a_{-\eps}$ with opposite signs, thus cancel out in $A_{n,m}(\tG)$. Hence 
$$A_{n,m}(G)=A_{n,m}(\tG).$$

Let $G$ and $\tG$ be  two Gauss diagrams that differ by an
application of $\Omega_3$, as shown in Figure \ref{fig:Reidem} ($G$ is on the left and $\tG$ is on the right).

Then there is a bijective correspondence between ascending states of
$G$ and $\tG$. This correspondence preserves the signs and the
combinatorics of the order in which the tracing enters and leaves
the neighborhood of these arrows. The table below summarizes this
correspondence.

$$\begin{array}{ccccccc}  \hline \\[-12pt] \hspace{-5pt}
  \kar{1t}{1tp}{1}{2}{3}{} & \hspace{-8pt}
  \kar{1t}{1tp}{1}{3}{2}{} & \hspace{-8pt}
  \kar{1t}{1tp}{2}{1}{3}{} & \hspace{-8pt}
  \kar{1l}{1lp}{1}{2}{3}{} & \hspace{-8pt}
  \kar{1l}{1lp}{1}{3}{2}{} & \hspace{-8pt}
  \kar{1l}{1lp}{2}{3}{1}{} & \hspace{-8pt}
  \kar{1r}{1rp}{3}{1}{2}{}
       \hspace{-5pt}\vspace{2pt} \\ \hline
\end{array}
$$

$$\begin{array}{ccccccc}  \hline \\[-12pt] \hspace{-5pt}
  \kar{1r}{1rp}{2}{1}{3}{} & \hspace{-8pt}
  \kar{1r}{1rp}{1}{2}{3}{} & \hspace{-8pt}
  \kar{2tl}{2lrp}{1}{2}{3}{}& \hspace{-8pt}
  \kar{2tr}{2lrp}{2}{1}{3}{}& \hspace{-8pt}
  \kar{2lr}{2trp}{1}{2}{3}{}& \hspace{-8pt}
  \kar{2lr}{2tlp}{1}{3}{2}{} & \hspace{-8pt}
  \kar{3}{3p}{1}{2}{3}{}
       \hspace{-5pt}\vspace{2pt} \\ \hline
\end{array}
$$

For a better understanding of this table, let us explain one of the
cases in details. Denote the top, left, and right arrows in the
fragment by $\a_t$, $\a_l$, and $\a_r$ respectively.

Consider a state $S\cup\a_l\cup\a_r$ of $G$ which contains two arrows of the fragment. The order of tracing the fragment depends on $S$. Only two orders of tracing may give an ascending state:

$$\textrm{States of}\quad G:\qquad\qquad\ris{-8}{-3}{30}{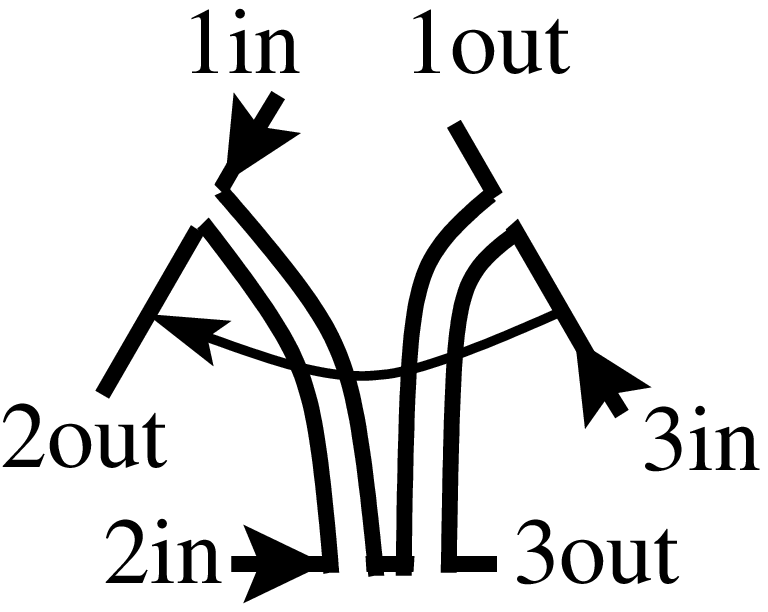}{-1.1}\quad\quad \textrm{or}\quad\quad\ris{-8}{-3}{30}{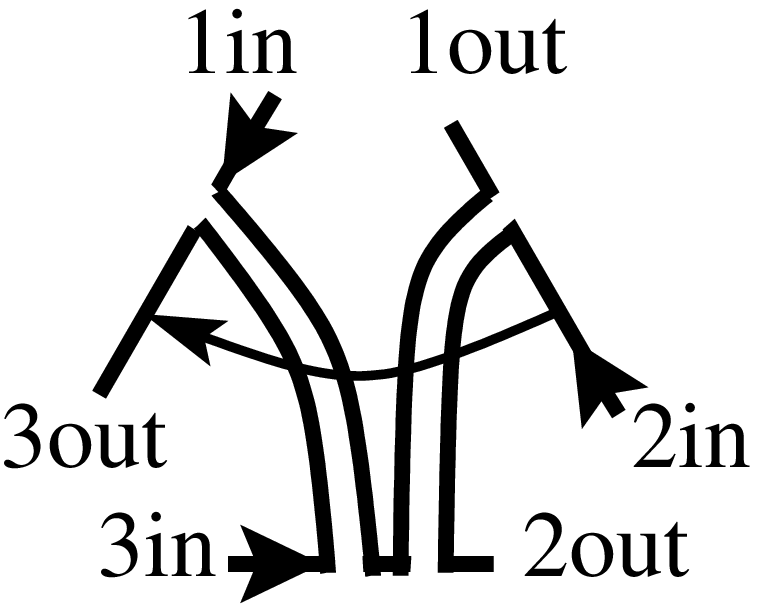}{-1.1}$$
$$\textrm{States of}\quad\tG:\qquad\qquad\ris{-8}{-3}{30}{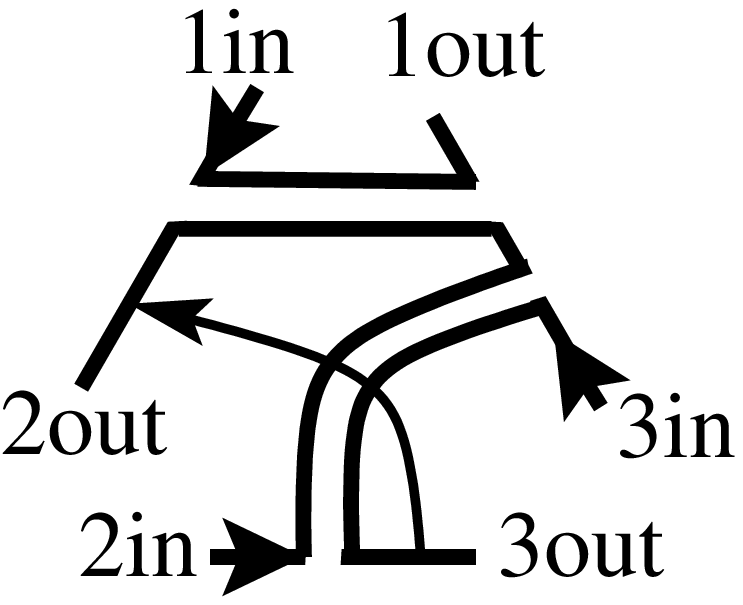}{-1.1}\quad\quad \textrm{or}\quad\quad\ris{-8}{-3}{30}{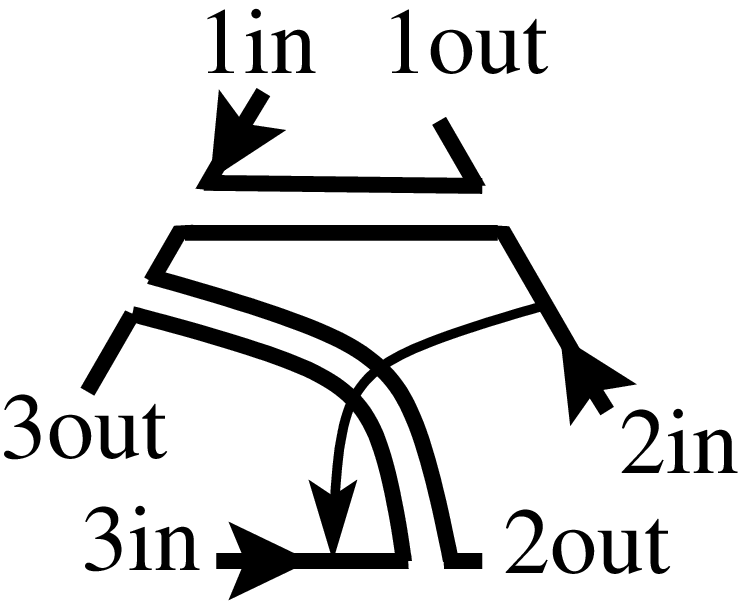}{-1.1}$$

Here the three consecutive entries and exits from the fragment are indicated by 1in, 1out, 2in, 2out, 3in, 3out.
In the first case, the corresponding state of $\tG$ is $S\cup\a_t\cup\a_r$. Note that the pattern of entries and exits from the fragment is indeed the same as in $G$.
In the second case, the corresponding state of $\tG$ is $S\cup\a_t\cup\a_l$. The pattern of entries and exits is again the same as in $G$.
\end{proof}

\subsection{Properties of $A_{n,m}(G)$ and $D_{n,m}(G)$}

\begin{thm}\label{thm:skein-relation}
Let $G_+$, $G_-$ and $G_0$ be Gauss diagrams which differ
only in the fragment shown in Figure \ref{fig:Conway}. Then
\begin{equation}
\nA(G_+)-\nA(G_-)=z\nA(G_0)\quad\quad \nD(G_+)-\nD(G_-)=z\nD(G_0).
\end{equation}
\end{thm}

\begin{figure}[htb]
\centerline{\includegraphics[height=0.7in]{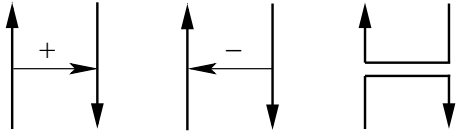}}
\caption{\label{fig:Conway} A Conway triple of Gauss diagrams.}
\end{figure}

\begin{proof}
Again we will prove this theorem for $\nA$. Denote by $m$ and
$m_0$ the number of circles in $G_\pm$ and $G_0$, respectively.
It is enough to prove that for each $n$ we have

\begin{equation}\label{eq:Conway}
A_{n,m}(G_+)-A_{n,m}(G_-)=A_{n-1,m_0}(G_0).
\end{equation}

Denote the arrows of $G_+$ and $G_-$ appearing in the fragment in Figure
\ref{fig:Conway} by $\a_+$ and $\a_-$, respectively. The proof of the skein
relation is the same as in \cite{CKR}. All ascending states of $G_\pm$ which
do not contain $\a_\pm$ cancel out in \eqref{eq:Conway} in pairs. Each
ascending state $S\cup\a_\pm$ of $G_\pm$ corresponds to a unique ascending
state $S$ of $G_0$, and vice versa: if $S$ is an ascending state of $G_0$, then
(depending on the order of the fragments in the tracing) exactly one of
$S\cup\a_+$ and $S\cup\a_-$ will give an ascending state on either $G_+$ or $G_-$.
\end{proof}

It is easy to see that both $A_{n,m}(G)$ and $D_{n,m}(G)$ depend on the
position of the base point when $G$ is a Gauss diagram associated with a
virtual link diagram. Let $G$ and $\widehat{G}$ be two Gauss diagrams shown 
in Figure \ref{fig:example-basepoint-virtual}. 
Then $A_{2,1}(G)=1$, $A_{2,1}(\widehat{G})=0$, $D_{2,1}(G)=0$, $D_{2,1}(\widehat{G})=1$.

However, for classical links this is not the case. Our next theorem states that in the case
of classical links, both $\nA(G)$ and $\nD(G)$ are independent of the position of the base point.

\begin{figure}[htb]
\centerline{\includegraphics[height=0.7in]{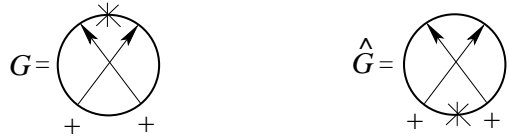}}
\caption{\label{fig:example-basepoint-virtual}Dependence on a basepoint.}
\end{figure}

\begin{thm}\label{thm:basepoint-independence}
Let $G$ be a based Gauss diagram of an $m$-component classical
link $L$. Both $\nA(G)$ and $\nD(G)$ are independent of the position of
the base point.
\end{thm}

\begin{proof}
We will prove the independence of $A_{n,m}(G)$ of the position of the base
point, the proof for $D_{n,m}(G)$ is the same. We prove this statement by
induction on the number of arrows in $G$.

If $G$ has no arrows, there is nothing to prove.
Now let us assume that the statement holds for any (classical) Gauss diagram
with less than $k$ arrows, and let $G$ be a Gauss diagram with $k$ arrows. If
$k<n$, then $A_{n,m}(G)=0$ and we are done, so we may assume that $k\ge n$.
Suppose that the base point lies on the $i$-th component of $G$. We should prove
that we may move the base point across any arrowhead or arrowtail on the $i$-th
component, and to shift it to any other, say, $j$-th, component.
Denote by $G_1,\ldots,G_7$ Gauss diagrams which differ only in a fragment
which looks like
$$\ris{-4}{-3}{135}{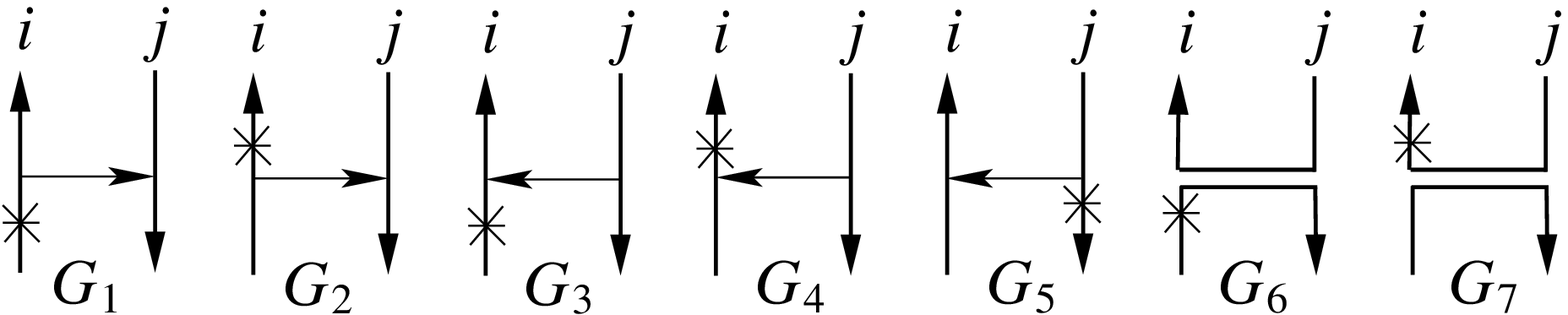}{-1.1}$$
respectively. It suffices to prove that we have:

\begin{equation}\label{eq:basept}
A_{n,m}(G_1)=A_{n,m}(G_2)\ ,\hspace{0.2cm}
A_{n,m}(G_3)=A_{n,m}(G_4)\ ,\hspace{0.2cm} A_{n,m}(G_3)=A_{n,m}(G_5).
\end{equation}

Denote by $\a$ the arrow appearing in the fragment above.
The first equality is immediate; indeed, in $G_1$ and $G_2$ there are no
ascending states with one boundary component which contain $\a$ (and
all other ascending states are in a bijective correspondence).

To prove the second equation in \eqref{eq:basept}, note that there is a
bijection between ascending states of $A_{n,m}(G_3)$ and $A_{n,m}(G_4)$
which do not contain $\a$; the remaining ascending states of $G_3$ and $G_4$
look exactly like the ones of $G_6$ and $G_7$, respectively; thus we have
$$A_{n,m}(G_3)-A_{n,m}(G_4)=A_{n-1,m_0}(G_6)-A_{n-1,m_0}(G_7)=0,$$
where $m_0=m\pm 1$ is the number of circles in $G_6$ and $G_7$ and the
second equality holds by the induction hypothesis. It is worth to mention that since $G_3$ and $G_4$ are diagrams
of classical links, then so are $G_6$ and $G_7$.

The proof of the last equality in \eqref{eq:basept} is more complicated.
We will use an inner induction on the number $r$ of arrows which have
only their arrowtail on the $i$-th component. If $r=0$, then $A_{n,m}(G_3)=A_{n,m}(G_5)=0$. Indeed, for $r=0$ there are no
ascending states in $G_5$ (since we cannot reach the $i$-th circle), so $A_{n,m}(G_5)=0$.
Also, since $r=0$, $i$-th component of the link is under all other components,
so we may move it apart by a finite sequence of Reidemeister moves $\Omega_2$
and $\Omega_3$ applied away from the base point, converting $G_3$ into a Gauss
diagram with an isolated $i$-th component  (here we use the fact that $G_3$ is
associated with a classical link diagram); thus $A_{n,m}(G_3)=0$ by Theorem
\ref{thm:Reidemeister-basepoint}.

Let's establish the step of induction. On both $G_3$ and $G_5$ pick the same
arrow, which has only its arrowtail on the $i$-th component, and apply the skein
relation of Theorem \ref{thm:skein-relation} to simplify the corresponding Gauss
diagrams. Diagrams on the right-hand side of the skein relation have less than
$k$ crossings, so the right-hand sides are equal by the induction on $k$; the
remaining terms in the left-hand side are also equal by induction on $r$.
\end{proof}

\begin{cor}\label{cor:Chmutov-Conway}
Let $G$ be a Gauss diagram of an $m$-component classical link $L$.
Then
$$\nA(G)=\nD(G)=\nabla(L),$$
i.e. for every $n$ we have $A_{n,m}(G)=D_{n,m}(G)=c_n(L)$.
\end{cor}

\begin{proof} 
By Theorems \ref{thm:Reidemeister-basepoint} -- \ref{thm:basepoint-independence},
$A_{n,m}(G)$ and $D_{n,m}(G)$ are link invariants which satisfy the same
skein relation as $c_n(L)$. It remains to compare their normalization, i.e. their
values on the unknot $O$. A standard Gauss diagram $G(O)$ of an
unknot consists of one circle with no arrows. It follows that $D_{0,1}(G(O))=A_{0,1}(G(O))=1$,
and $D_{n,m}(G(O))=A_{n,m}(G(O))=0$ otherwise. Hence $A_{n,m}(G)$ and
$D_{n,m}(G)$ have the same normalization as $c_n(L)$ and the corollary follows.
\end{proof}

\begin{ex}\rm
Consider a $3$-component link $L$ and the Gauss diagram $G$ of $L$
shown below:
$$\ris{+1}{-1}{30}{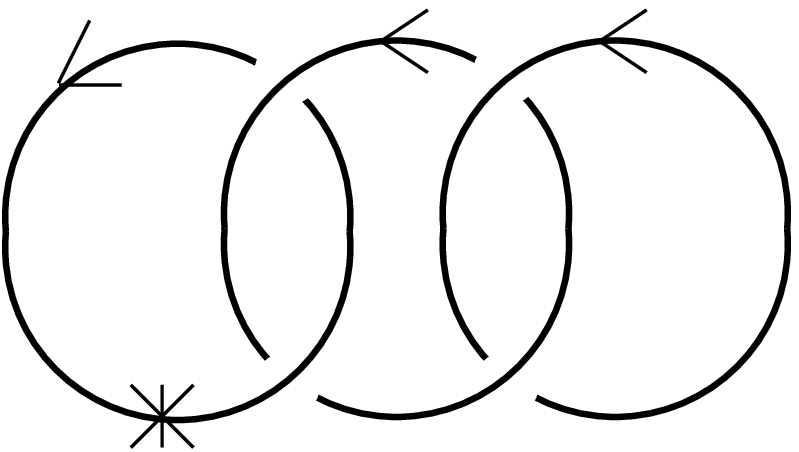}{-1.1}\quad\quad\quad\quad
\ris{+1}{-1}{53}{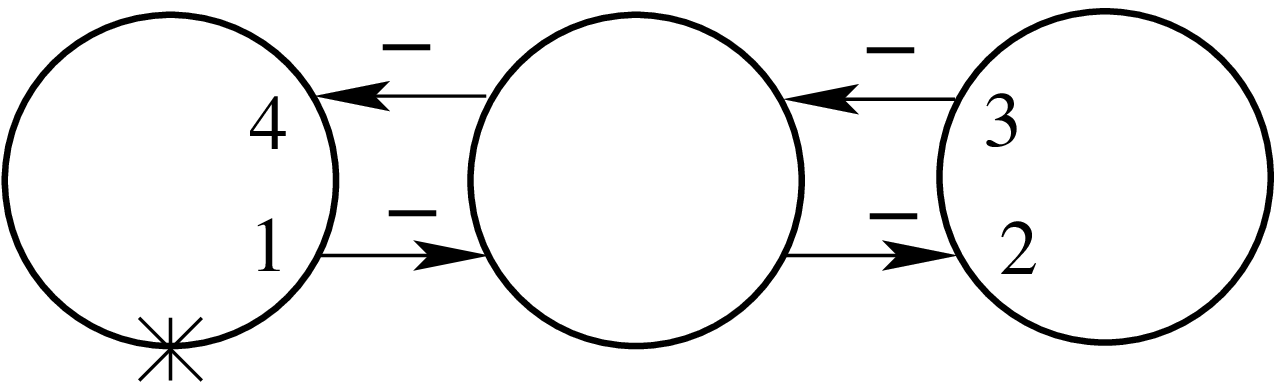}{-1.1}$$
The only ascending state of $G$ is $\{3,4\}$. It sign is $+$. Thus
$c_2(L)=1$ and $c_n(L)=0$ for all $n\neq 2$, so $\nabla(L)=z^2$.
\end{ex}

\subsection{Alexander-Conway polynomials of long virtual links}
In this subsection we study properties of the polynomials $\nA$ and $\nD$
for long virtual links, and compare them with other existing constructions.

Let $L$ be a classical or long virtual link and $G$ be any Gauss diagram of $L$.
Polynomials $\nA(L):=\nA(G)$ and $\nD(L):=\nD(G)$ were defined in \cite{CKR},
but the proof that they are well defined for classical links with more than $2$
components and for long virtual links was not presented. By Theorem
\ref{thm:Reidemeister-basepoint} and Corollary \ref{cor:Chmutov-Conway},
$\nA(L)$ and $\nD(L)$ are invariants of $L$, and if $L$ is a classical link
$$\nA(L)=\nD(L)=\nabla(L).$$
Note that for long virtual links it may happen that
$$\nA(L)\neq\nD(L).$$
For example, let $G$ be a Gauss diagram of the long virtual Hopf link $L$
shown in Figure \ref{fig:trefoil+hopf+virtual}. Then $\nA(L)=z$, but $\nD(L)=0$.
We denote by
$$\nAD(L):=\nA(L)-\nD(L).$$
This polynomial vanishes on classical links but, as we will see below,
may be used to distinguish virtual links from classical links.

Let $D$ be a diagram of a (long) virtual link $L$ and $G$ its corresponding
Gauss diagram. Pick a classical crossing on $D$. A move on $D$ and the
corresponding move on $G$ shown in Figure \ref{fig:Virtualization-move}
is called the \textit{virtualization} move.

\begin{figure}[htb]
\centerline{\includegraphics[height=1.2in]{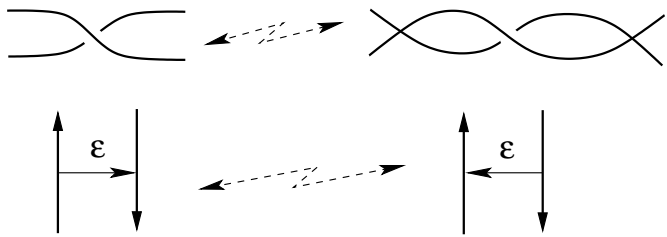}}
\caption{\label{fig:Virtualization-move} The virtualization move.}
\end{figure}

\begin{thm} \label{thm:Vassiliev-GPV}
Let $L$ and $L_1$ be long virtual links. Then the following holds.
\begin{enumerate}
\item
$\nA(L\#L_1)=\nA(L)\nA(L_1)\ ,\quad \nD(L\#L_1)=\nD(L)\nD(L_1)$, where $L\#L_1$ denotes a long virtual link which is a connected sum of $L$ and $L_1$.
\item
Non-trivial coefficients of $\nD$ and $\nA$ are not invariant under the virtualization move
\item
Coefficients of $\nD(L)$ and $\nA(L)$ are Vassiliev invariants in the
sense of both GPV \cite{GPV} and Kauffman \cite{Ka}.
\end{enumerate}
\end{thm}

\begin{proof} The proof of $(1)$ is straightforward and
follows from the definition of $\nA$ and $\nD$.

Now we prove $(2)$. Consider even $n$ first. Let $n=2k$ and let $G$
and $G_1$ be Gauss diagrams of long virtual knots shown in Figure
\ref{fig:Virtualization-move-Gauss}a and  \ref{fig:Virtualization-move-Gauss}b
respectively. Note that $G$ and $G_1$ differ by an application of the
virtualization move. Both $G$ and $G_1$ have $n+1$ arrows. It is
easy to check that
$$A_{n,1}(G)=-1  \quad D_{n,1}(G)=1, \quad \textrm{but}\quad A_{n,1}(G_1)=D_{n,1}(G_1)=0.$$

\begin{figure}[htb]
\centerline{\includegraphics[height=2.2in]{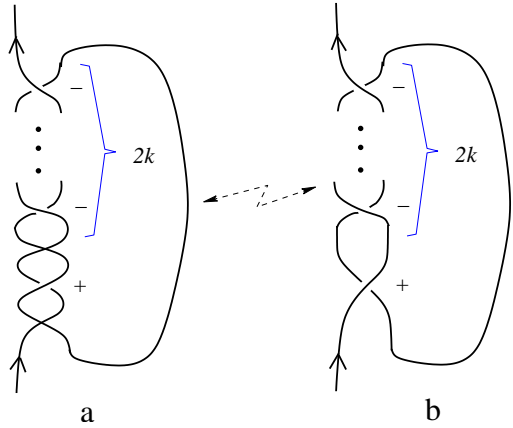}}
\caption{\label{fig:Virtualization-move-Gauss}Diagrams of long virtual
knots that differ by an application of the virtualization move.}
\end{figure}

To prove the statement for odd $n$, add a Hopf-linked unknot to the above diagrams.

Now we prove $(3)$. It is enough to prove that $A_{n,m}(G)$ and $D_{n,m}(G)$
are GPV finite type invariants, because any GPV finite type invariant is
automatically of Kauffman finite type. But \cite[page 12]{GPV} implies,
that any invariant given by an arrow diagram formula with $n$ arrows is
GPV finite type of degree $n$.
\end{proof}

\begin{rem}\rm
It follows from \cite[Theorem 1.1]{C}, that $(2)$ in Theorem
\ref{thm:Vassiliev-GPV} follows from $(3)$.
\end{rem}

Let $K_T$ be a virtual knot with a virtual diagram shown in Figure
\ref{fig:Kishino-knot}. This knot is called the Kishino knot. It has attracted
attention for its remarkable property that it is a connected sum of two
diagrams of the trivial knot; it has trivial Jones polynomial, $Z'_{K_T}(t,y)=0$
(for the definition of $Z'$ and its properties see \cite{Saw} and Paragraph 3.4), and the virtual knot group of $K_T$
is isomorphic to $\Z$, see \cite{Ki}. It was first proved to be non-classical
in \cite{Ki}. We show that $K_T$ is non-classical using the polynomials
$\nA$, $\nD$ and $\nAD$.

\begin{figure}[htb]
\centerline{\includegraphics[height=0.8in]{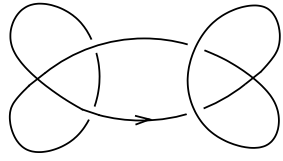}}
\caption{\label{fig:Kishino-knot} A diagram of the Kishino knot.}
\end{figure}

\begin{prop} \label{prop:Kishino-knot}
Polynomials $\nA$, $\nD$ and $\nAD$ detect
the fact that $K_T$ is a non-classical knot.
\end{prop}

\begin{proof}
Recall that for any Gauss diagram $G$ of a classical knot all these
polynomials are independent of the position of the base point.
Consider two Gauss diagrams $G$ and $\widehat{G}$ of $K_T$ which differ
only by the position of the base point, see Figure \ref{fig:Kishino-Gauss}.

\begin{figure}[htb]
\centerline{\includegraphics[height=0.7in]{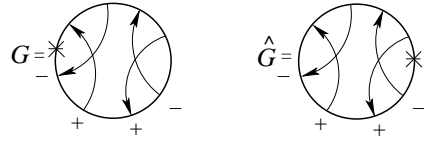}}
\caption{\label{fig:Kishino-Gauss} Gauss diagrams of the Kishino knot.}
\end{figure}

We have
$$\nA(G)=1-2z^2+z^4\quad \nD(G)=1\quad \nAD(G)=-2z^2+z^4\quad \textrm{but}$$
$$\nA(\widehat{G})=1\quad \nD(\widehat{G})=1-2z^2+z^4\quad \nAD(\widehat{G})=2z^2-z^4.$$
\end{proof}

\begin{rem}\rm
In order to prove Proposition \ref{prop:Kishino-knot}, i.e. to prove that the Kishino knot is not classical, it is enough to use Corollary \ref{cor:Chmutov-Conway}, i.e. to show that $\nAD(G)\neq 0$, where $G$ is a Gauss diagram shown in Figure \ref{fig:Kishino-Gauss}.
\end{rem}

\subsection{Comparison with other constructions of
Alexander-Conway polynomials of virtual links}

In \cite{Saw} Sawollek associated to every link diagram $D$ of a virtual
link $L$ a Laurent polynomial $Z'_D(t,y)$ in two variables $t,y$. He proved
that $Z'_D(t,y)$ is an invariant of virtual links up to multiplication by powers
of  $t^{\pm 1}$, and that it vanishes on classical links. He also showed that
$Z'_D(t,y)$ satisfies the following skein relation:
$$Z'_{D_+}(t,y)-Z'_{D_-}(t,y)=(t^{-1}-t)Z'_{D_0}(t,y),$$
where $D_+, D_-, D_0$ is a Conway triple of diagrams shown in Figure \ref{fig:triple}.
It is obvious that both $\nA(L)$ and $\nD(L)$ are crucially different from
$Z'_L(t,y)$ because both of them do not vanish on classical links, but one
can suspect that $\nAD(L)$ coincides with $Z'_L(t,y)$ after a possible
renormalization and a change of variables $z=t^{-1}-t$. Sawollek proved
the following theorem:

\begin{thm}[\cite{Saw}]\label{thm:disconnected-Sawolik}
Let $D$, $D_1$, $D_2$ be virtual link diagrams and let $D_1\sqcup D_2$
denote the disconnected sum of the diagrams $D_1$ and $D_2$. Then
$$Z'_{D_1\sqcup D_2}(t,y)=Z'_{D_1}(t,y)Z'_{D_2}(t,y).$$
\end{thm}

Note that $\nAD(L\sqcup L)=0$ for any long virtual link $L$, but
$Z'_{L\sqcup L}(t,y)=(Z'_L(t,y))^2$. Thus $\nAD(L)$ is also different from $Z'_L(t,y)$.

Other generalizations of Alexander polynomials to virtual links
are derived from the virtual and extended virtual link groups,
see \cite{Saw} and \cite{SW, SW1} respectively.

\textbf{1.} Following \cite{Saw} we denote by $\Delta_L(t)$ a polynomial
which is derived from the virtual link group of a link $L$. It is well defined
up to sign and multiplication by powers of $t^{\pm 1}$. For every virtual
link diagram $D$ the associated polynomial is denoted by $\Delta_D(t)$.
In contrast to the classical Alexander polynomial, the Alexander polynomial
of \cite{Saw} for virtual links does not satisfy any linear skein relation as
stated in the next theorem:

\begin{thm}[\cite{Saw}]\label{thm:no-skein-Sawolik}
For any normalization $A_D(t)$ of the polynomial $\Delta_D(t)$, i.e.,
$A_D(t)=\eps_Dt^{n_D}\Delta_D(t)$ with some $\eps_D\in\{-1,1\}$ and
$n_D\in\Z$, the equation
$p_1(t)A_{D_+}(t) + p_2(t)A_{D_-}(t) + p_3(t)A_{D_0}(t) = 0$ with
$p_1(t),p_2(t),p_3(t)\in\Z[t^{\pm 1}]$ has only the trivial solution
$p_1(t) = p_2(t) = p_3(t) = 0$.
\end{thm}

Since $\nA$, $\nD$ and $\nAD$ satisfy the Conway skein relation, it follows
that all of them are different from $\Delta_D(t)$ of \cite{Saw}.

\textbf{2.} Let $L$ be a virtual link. The polynomial $\Delta_1(L)$ of
\cite{SW, SW1} is a polynomial in variables $v,u$ and is well defined up to
multiplication by powers of $(uv)^{\pm 1}$. If $L$ is a classical link,
then $\Delta_1(L)$ is equal to the Alexander polynomial in the variable $uv$.
It follows that $\nAD$ is different from $\Delta_1$, because $\Delta_1$ is not
identically zero on the family of classical links.

Given a virtual link $L$, we denote by $L^*$ the mirror
image of $L$, i.e. a link obtained by inverting the sign of each classical
crossing in a diagram of $L$. The following corollary was proved in \cite{SW1}.

\begin{cor}[\cite{SW1}, Corollary 5.2]\label{cor:mirror-SW}
Let $L$ be a virtual $m$-component link. Then
$$\Delta_1(L)(u,v)=(-1)^m\Delta_1(L^*)(v,u)$$
up to multiplication by powers of $(uv)^{\pm 1}$.
\end{cor}

In particular, for any virtual knot $\Delta_1(K)(u,v)=-\Delta_1(K^*)(v,u)$.

Consider a mirror pair of long virtual knots $K$ and $K^*$ with Gauss
diagrams $G$ and $G^*$ shown in Figure \ref{fig:mirror-pair}.

\begin{figure}[htb]
\centerline{\includegraphics[height=0.7in]{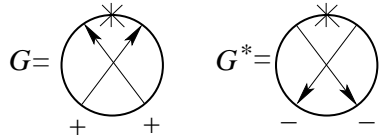}}
\caption{\label{fig:mirror-pair}A mirror pair of Gauss diagrams.}
\end{figure}

Then $\nA(K)=\nD(K^*)=1+z^2$, $\nA(K^*)=\nD(K)=1$. Thus both
$\nA$ and $\nD$ are different from $\Delta_1(L)$ in \cite{SW, SW1}.

Another way to see that both $\nA$ and $\nD$ are different from $\Delta_1$ is
by finding a virtual knot $K$, such that both $\nA$ and $\nD$ detect that this
knot is non-classical, but $\Delta_1(K)=1$ (so $\Delta_1$ does not distinguish
this knot from the unknot). An example of such a knot $K$, together with a pair
$G$ and $\widehat{G}$ of its Gauss diagrams which differ by the position of the
base point, is given in Figure \ref{fig:virtual-nontrivial-knot}. It was shown in
\cite{SW1} that $\Delta_1(K)=1$.

\begin{figure}[htb]
\centerline{\includegraphics[height=0.8in]{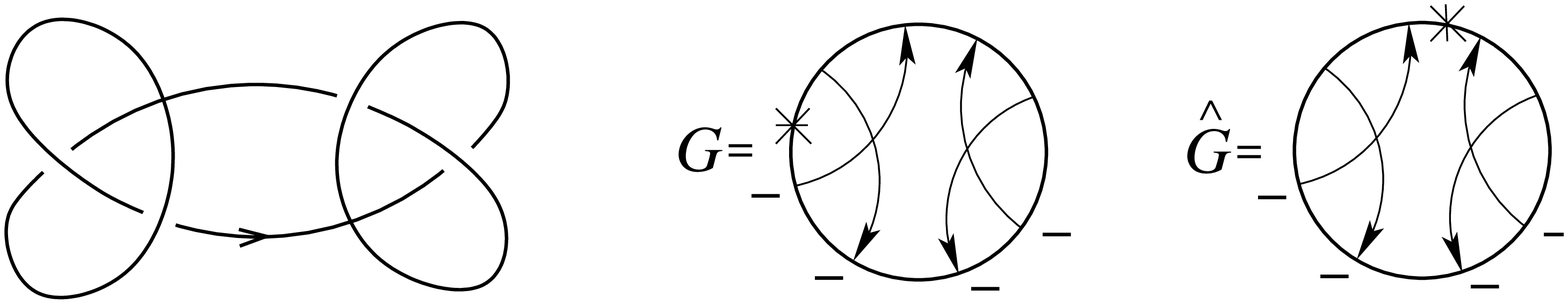}}
\caption{\label{fig:virtual-nontrivial-knot} A virtual knot and two of its based Gauss diagrams.}
\end{figure}

We have
\begin{align*}
&\nA(G)=1+z^2\hspace{2cm} &\nD(G)=1+z^2\qquad \textrm{but}\\
&\nA(\widehat{G})=1+2z^2+z^4 &\nD(\widehat{G})=1.\hspace{2.2cm}
\end{align*}

It follows that both $\nA$ and $\nD$ show that $K$ is non-classical.

Finally, another generalization of the Alexander polynomial (related to the
polynomial $Z'$ of \cite{Saw}) to long virtual knots was presented in \cite{A}.
It is a Laurent polynomial $\zeta$ in a variable $t$ over the following ring
$T=\mathbb{Z}[p,p^{-1},q,q^{-1}]/((p-1)(p-q),(q-1)(p-q))$. This polynomial
also vanishes on classical knots and thus $\zeta$ significantly differs from
$\nA$ and $\nD$.

\noindent\textbf{Question}. Is it possible to derive $\nAD$ from $\zeta$?

\section{Counting surfaces with two boundary components}
\label{sec-2-comp}

In this section we present a new infinite family of Gauss diagram formulas,
which correspond to counting of orientable surfaces with two boundary
components. At the end of this section we identify the resulting invariants
with certain derivatives of the HOMFLYPT polynomial.

\subsection{Link invariants and diagrams with two boundary components}
In this subsection we define invariants of classical links using
ascending and descending arrow diagrams with two boundary components.

Recall that for every Gauss diagram $G$ we defined notions of ascending and
descending separating states of $G$, see Subsection \ref{subsec-Enhanced-states}.
Also, for every ascending (respectively descending) separating state $S$ of $G$
and for an arrow diagram $A\in\A_{n,m}^2$ (respectively $A\in\D_{n,m}^2$)
we defined, in the same subsection, a notion of $S$-admissible pairing $\AG_S$.
Note that every ascending (respectively descending) separating state $S$ of $G$
defines two Gauss diagrams $G'_S$ and $G''_S$ as follows: $G'_S$ (respectively
$G''_S$) consists of all circles of $G_S$ labeled by 1 (respectively by 2), and its
arrows are arrows of $G$ with both ends on these circles. All arrows with ends
on circles of $G_S$ with different labels are removed. The base point on $G'_S$ is
the base point $*$ of $G$. The base point on $G''_S$ is placed near the first arrow
in $S$ which we encounter as we walk on $G$ starting from~$*$.

If $G$ is a Gauss diagram of a classical link $L$, then $G'_S$ and $G''_S$
correspond to classical links $L'_S$ and $L''_S$, which are defined as follows.
We smooth all crossings which correspond to arrows in $S$, as shown below:

\begin{equation*}
\ris{-4}{-3}{60}{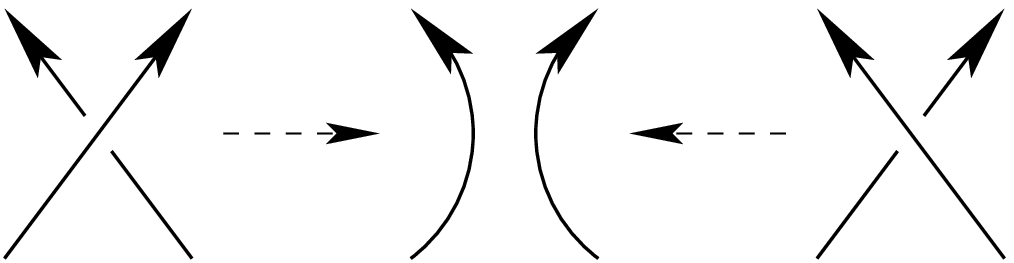}{-1.1}
\end{equation*}

We obtain a diagram of a smoothed link $L_S$ with labeling of components
induced from the labeling of circles of $G_S$. Denote by $L'_S$ and $L''_S$
sublinks which consist of components labeled by $1$ and $2$ respectively. Let
$m'$ and $m''$ denote the number of components of $L'_S$ and $L''_S$ respectively, and as usual let $G'_S, G''_S$ be the corresponding Gauss diagrams.  It is an immediate consequence of Definition \ref{D:G_S} that

\begin{equation*}
A^2_{n,m}(G)_S=\s(S)\sum_{i=0}^{n-|S|}A_{i,m'}(G'_S)A_{n-|S|-i,m'}(G''_S)
\end{equation*}

\begin{equation*}
D^2_{n,m}(G)_S=\s(S)\sum_{i=0}^{n-|S|}D_{i,m'}(G'_S)D_{n-|S|-i,m'}(G''_S),
\end{equation*}
where $|S|$ is the number of arrows in $S$ and $\s(S)=\prod_{\a\in S}\s(\a)$.

It follows from Corollary \ref{cor:Chmutov-Conway}
that for every $n\geq 0$ we have
$$A_{n,m'}(G'_S)=D_{n,m'}(G'_S)=c_n(L'_S)\quad \textrm{and}\quad A_{n,m''}(G''_S)=D_{n,m''}(G''_S)=c_n(L''_S)$$. 
As an immediate corollary we get

\begin{lem}\label{lem:essential-arrows}
Let $G$ be a Gauss diagram of an $m$-component link $L$.
Then for every $n\geq 0$ and an ascending (respectively descending)
separating state $S$ of $G$ we have
\begin{equation*}
A^2_{n,m}(G)_S=\s(S)\sum_{i=0}^{n-|S|}c_i(L'_S)c_{n-|S|-i}(L''_S)
\end{equation*}
\begin{equation*}
D^2_{n,m}(G)_S=\s(S)\sum_{i=0}^{n-|S|}c_i(L'_S)c_{n-|S|-i}(L''_S).
\end{equation*}
\end{lem}

Summing over all ascending (descending) separating states $S$, we obtain

\begin{cor}\label{cor:enhanced-essential-arrows}
Let $G$ be any Gauss diagram of an $m$-component
link $L$. Then for every $n\geq 0$ we have
\begin{align*}
& A^2_{n,m}(G)=\sum_{k=1}^n\sum_{S,|S|=k}\s(S)\sum_{i=0}^{n-k}c_i(L'_S)c_{n-k-i}(L''_S) \\
& D^2_{n,m}(G)=\sum_{k=1}^n\sum_{S,|S|=k}\s(S)\sum_{i=0}^{n-k}c_i(L'_S)c_{n-k-i}(L''_S)
\end{align*}
where the second summation is over all ascending and descending
separating states $S$ of $G$ respectively.
\end{cor}

It turns out that both $A^2_{n,m}(G)$ and  $D^2_{n,m}(G)$ are
invariant under $\Omega_2$ and $\Omega_3$:

\begin{thm}\label{thm:invariance-2-3}
Let $G$ be any Gauss diagram of an $m$-component
link $L$. Then $A^2_{n,m}(G)$ and  $D^2_{n,m}(G)$ are
invariant under Reidemeister moves $\Omega_2$ and $\Omega_3$
which do not involve the base point.
\end{thm}

\begin{proof}
We will prove the invariance of $A^2_{n,m}(G)$; the proof for $D^2_{n,m}(G)$ is the same.

Let $G$ and $\tG$ be two Gauss diagrams that differ by an application
of $\Omega_{2}$, so that $\tG$ has two additional arrows $\a_\eps$
and $\a_{-\eps}$, see Figure \ref{fig:Reidem}.
Ascending states of $G$ are in bijective correspondence with
ascending states of $\tG$ which do not contain $\a_{\pm\eps}$.
Note that $\a_\eps$ and $\a_{-\eps}$ can not be both in the image
of $\phi:A\to G$ with $A\in \mathcal{A}^2_{n,m}$, because $A$ is an
ascending diagram with two boundary components. Ascending states
of $\tG$ which contain one of $\a_{\pm\eps}$ come in pairs $S\cup\a_\eps$
and $S\cup\a_{-\eps}$ with opposite signs, thus cancel out in $\AA_{n,m}(\tG)$.
Hence $$A^2_{n,m}(G)=A^2_{n,m}(\tG).$$

Now, let $G$ and $\tG$ be  two Gauss diagrams that differ by an application
of $\Omega_3$, see Figure \ref{fig:Reidem} ($G$ is on the left and $\tG$ is
on the right). Denote the top, left, and right arrows in the fragment by $\a_t$,
$\a_l$, and $\a_r$ respectively. There is a bijective correspondence between
ascending separating states of $G$ and $\tG$, such that none of the arrows
$\a_r$, $\a_l$ and $\a_t$ belong to these states. Indeed, we may identify
separating states of $G$ and $\tG$ which have the same arrows and
the same labeling of arcs. For any such separating state $S$ we have
$A^2_{n,m}(G)_S=A^2_{n,m}(\tG)_S$ by Lemma \ref{lem:essential-arrows}
and Theorem \ref{thm:Reidemeister-basepoint}.

An ascending separating state of $G$ or $\tG$ may contain either exactly
one arrow of the fragment i.e. $\a_r$ or $\a_l$ or $\a_t$, or it may contain
both arrows $\a_r$ and $\a_l$. There is a bijective correspondence between
ascending separating states of $G$ and $\tG$ which contain $\a_l$. Two
possible cases of this correspondence (which differ by the labeling) are
shown in Figure \ref{fig:2comp-separating-lr} and Figure \ref{fig:2comp-separating-lr1}. Abusing the notation we denote the corresponding ascending
separating states by $S_r$, $S_l$, $S_t$, $S_{lr}$, and $\widetilde{S}_r$, $\widetilde{S}_l$, $\widetilde{S}_t$, $\widetilde{S}_{lr}$.

\begin{figure}[htb]
\centerline{\includegraphics[height=2in]{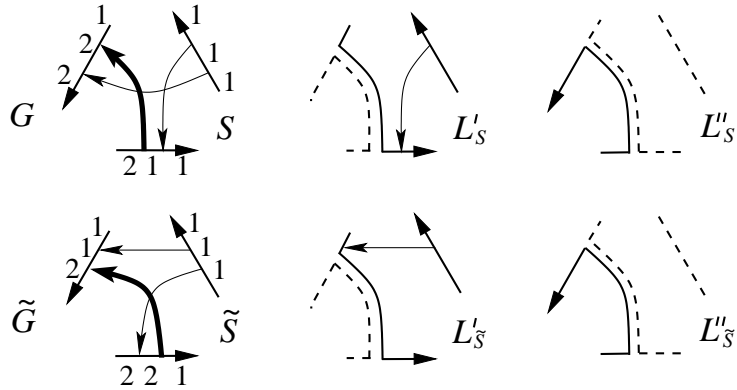}}
\caption{\label{fig:2comp-separating-lr} Identifying ascending separating
states containing $\a_l$.}
\end{figure}

\begin{figure}[htb]
\centerline{\includegraphics[height=2in]{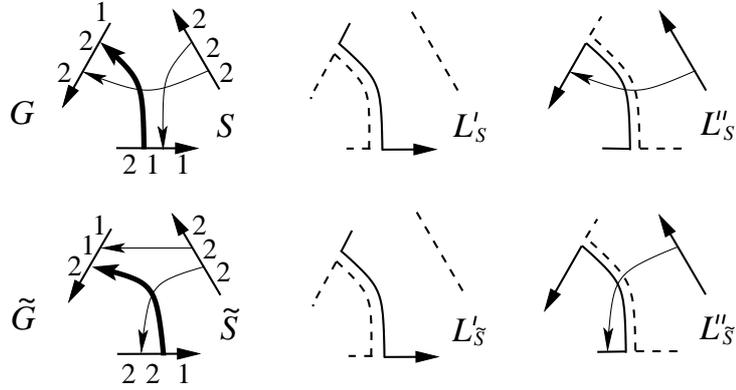}}
\caption{\label{fig:2comp-separating-lr1} Identifying other ascending
separating states containing $\a_l$.}
\end{figure}

In both cases, links $L'_S$ and $L''_S$ constructed from $G$ and $\tG$
are isotopic, thus by Lemma \ref{lem:essential-arrows}
$A^2_{n,m}(G)_S=A^2_{n,m}(\tG)_{\widetilde{S}}$.
The situation with ascending separating states which contain $\a_r$ is
completely similar and is omitted.

The correspondence of ascending separating states which contain
$\a_l\cup\a_r$ or $\a_t$ is more complicated. One of the two possible
cases is summarized in Figure \ref{fig:2-comp-separating-tlr}.

\begin{figure}[htb]
\centerline{\includegraphics[height=3in]{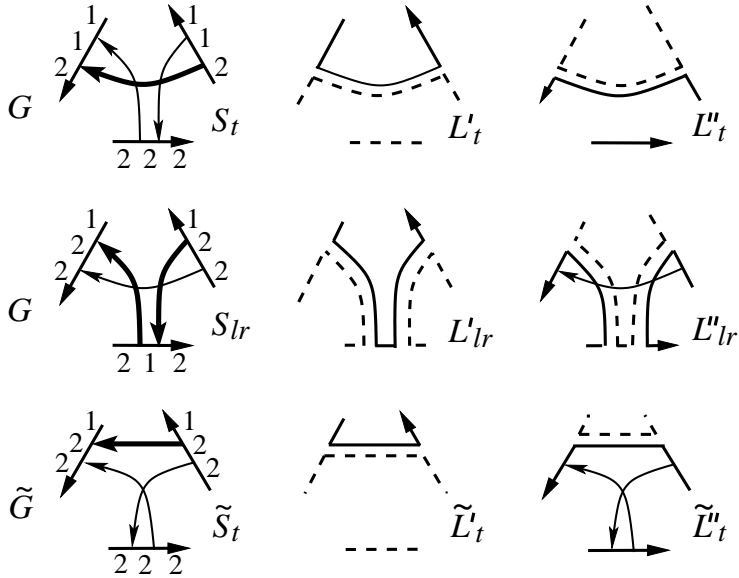}}
\caption{\label{fig:2-comp-separating-tlr} Comparison of ascending
separating states of $G$ and $\tG$.}
\end{figure}

\begin{figure}[htb]
\centerline{\includegraphics[height=3in]{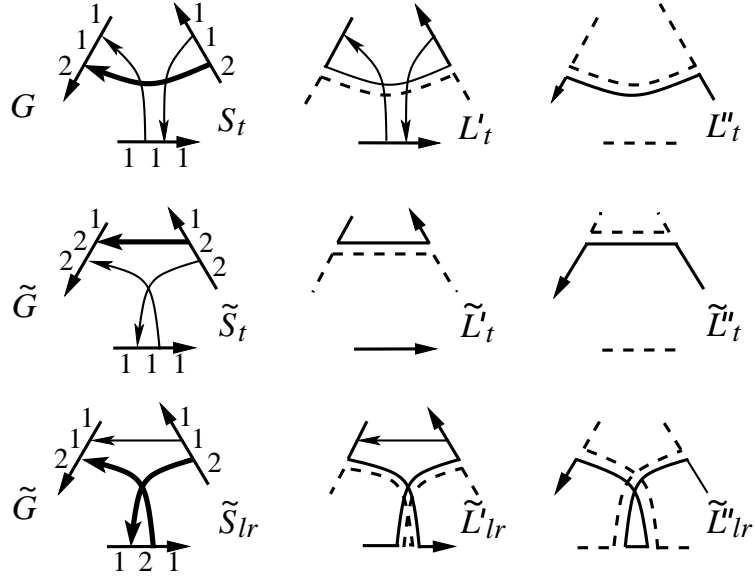}}
\caption{\label{fig:2-comp-separating-tlr1}Comparison of other
ascending separating states of $G$ and $\tG$.}
\end{figure}

Links $\widetilde{L'}_t$, $L'_t$ and $L'_{lr}$ are isotopic, thus
$$c_i(\widetilde{L'}_t)=c_i(L'_t)=c_i(L'_{lr}).$$
For $c_i(\widetilde{L}^{''}_t)$ we have:
$$c_i\left(\ris{-5}{-1}{15}{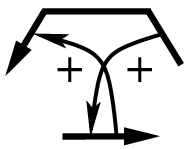}{-1}\right)=
c_i\left(\ris{-5}{-1}{15}{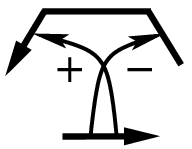}{-1}\right)+c_{i-1}\left(\ris{-5}{-1}{15}{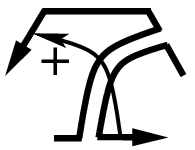}{-1}\right)=
c_i\left(\ris{-5}{-1}{15}{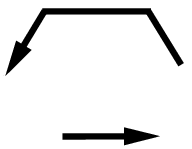}{-1}\right)+c_{i-1}\left(\ris{-5}{-1}{15}{c0_2comp}{-1}\right)\hspace{1mm}.$$
The first equality is the skein relation of Theorem \ref{thm:skein-relation},
and the second equality holds by the invariance of $c_i$ under $\Omega_2$.
Hence 
$$c_i(\widetilde{L}''_t)=c_i(L''_t)+c_{i-1}(L''_{lr}).$$  
Denote by $k$ the number of arrows in $S_t$, where $S_t$ is shown in Figure \ref{fig:2-comp-separating-tlr}. 
Note that the number of arrows in $\widetilde{S}_t$ and $S_{lr}$ is $k$ and $k+1$, respectively. Thus

\begin{equation*}
\sum_{i=0}^{n-k}c_i(\widetilde{L}'_t)c_{n-k-i}(\widetilde{L}''_t)=
\sum_{i=0}^{n-k}c_i(L'_t)c_{n-k-i}(L''_t)+
\sum_{i=0}^{n-k-1}c_i(L'_{lr})c_{n-k-i-1}(L''_{lr}).
\end{equation*}

Note that $\s(\widetilde{S}_t)=\s(S_t)=\s(S_{lr})$, thus by Lemma \ref{lem:essential-arrows}
$$A^2_{n,m}(\tG)_{\widetilde{S}_t}=A^2_{n,m}(G)_{S_t}+A^2_{n,m}(G)_{S_{lr}}.$$

The second possible case (which differs by labeling) is shown in
Figure \ref{fig:2-comp-separating-tlr1}.
Abusing the notation we again denote the corresponding ascending
separating states by $S_t$, $\widetilde{S}_t$, $\widetilde{S}_{lr}$.
Links $L''_t$, $\widetilde{L}''_t$ and $\widetilde{L}''_{lr}$
are isotopic. Applying the skein relation for $c_i(L'_t)$ similarly to the
above, in this case we get
$$A^2_{n,m}(G)_{S_t}=A^2_{n,m}(\tG)_{\widetilde{S}_t}+A^2_{n,m}(\tG)_{\widetilde{S}_{lr}}.$$
\end{proof}

Our next step is to study the behavior of $A^2_{n,m}(G)$ and
$D^2_{n,m}(G)$ under an application of the move
$\Omega_1$. Both $A^2_{n,m}(G)$ and  $D^2_{n,m}(G)$ change
under $\Omega_1$, see Example \ref{ex:R1-noninvariance}.

\begin{ex}\label{ex:R1-noninvariance}\rm
Let $G$, $G_1$ and $G_2$ be Gauss diagrams of an unknot shown
in Figures \ref{fig:Om1_noninv}a,  \ref{fig:Om1_noninv}b and
\ref{fig:Om1_noninv}c respectively. Then $\DD_{1,1}(G)=0$, but
$\DD_{1,1}(G_1)=1$; and $\AA_{1,1}(G)=0$, but $\AA_{1,1}(G_2)=-1$.

\begin{figure}[htb]
\centerline{\includegraphics[height=0.9in]{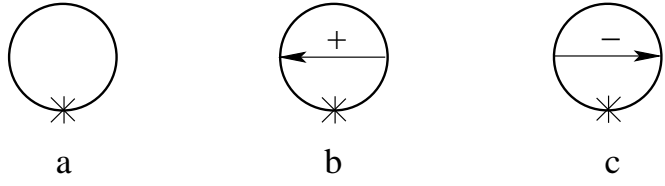}}
\caption{\label{fig:Om1_noninv}Gauss diagrams $G$, $G_1$ and $G_2$.}
\end{figure}
\end{ex}

However, this problem is easy to solve. Denote by
$$AD_{n,m}(G):=A^2_{n,m}(G)+D^2_{n,m}(G)$$ 
and let
$$I_{n,m}(G):=AD_{n,m}(G)-w(G)c_{n-1}(L),$$
where $w(G)$ is the writhe of $G$, i.e. the sum of signs of all arrows in $G$.

\begin{thm}\label{thm:invariance}
Let $G$ be a Gauss diagram of a long $m$-component
link $L$. Then $I_{n,m}(G)$ is an invariant of an underlying link
$L$, i.e. is independent of a choice of $G$.
\end{thm}

\begin{proof}
By Theorem \ref{thm:invariance-2-3}, it remains to prove the
invariance of $I_{n,m}$ under $\Omega_{1}$ (applied away from
the base point). Let $\tG$ and $G$ be two Gauss diagrams which
are related by an application of $\Omega_{1}$, such that $\tG$
contains a new isolated arrow $\a$. Then $\a$ contributes either
a new ascending or a new descending separating state $\{a\}$,
depending on whether we meet its head or tail first on the passage
from the base point. Contribution of this state to $AD_{n,m}(\tG)$
is either $\s(\a)A_{n-1,m}(G)$, or $\s(\a)D_{n-1,m}(G)$; but
$$\s(\a)A_{n-1,m}(G)=\s(\a)c_{n-1}(L)=\s(\a)D_{n-1,m}(G)$$
by Corollary \ref{cor:Chmutov-Conway}, thus
$$I_{n,m}(\tG)-I_{n,m}(G)=(\s(\a)-w(\tG)+w(G))c_{n-1}(L).$$
It remains to note that $w(\tG)-w(G)=\s(\a)$.
\end{proof}

Our next step is to study dependence of $A^2_{n,m}(G)$ and
$D^2_{n,m}(G)$ on the position of the base point. The example
below shows that each of them depends on the base point.

\begin{ex} \rm
Let $G$ and $\widehat{G}$ be two Gauss diagrams of the right-handed
trefoil shown below. Then $A_{3,1}^2(G)=0$ and $D_{3,1}^2(G)=1$,
but $A_{3,1}^2(\widehat{G})=1$ and $D_{3,1}^2(\widehat{G})=0$.
$$\ris{-4}{-3}{65}{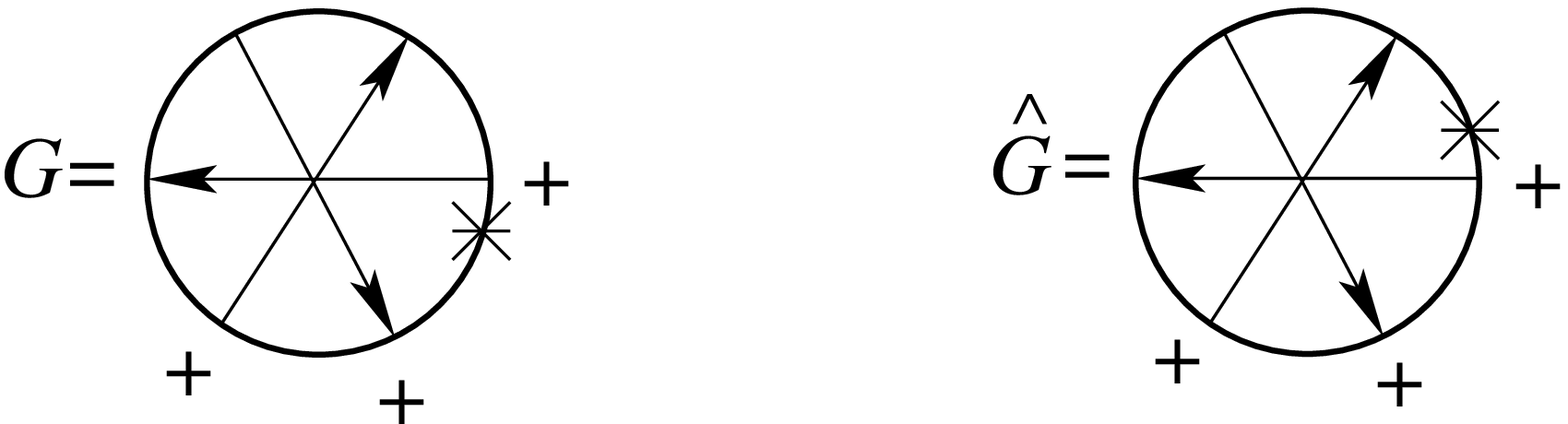}{-1.1}$$
\end{ex}

However, it turns out that the sum
$A^2_{n,m}(G)+D^2_{n,m}(G)=AD_{n,m}(G)$ does not depend
on the base point. Indeed, let $G$ and $\widehat{G}$ be two Gauss
diagrams which differ only by a position of their base points. Let
$\widehat{S}$ be an ascending separating state of $\widehat{G}$.
If an arc which contains the base point $*$ of $G$ is labeled by $1$,
then $S=\widehat{S}$ is an ascending separating state of $G$ and
by Lemma \ref{lem:essential-arrows} we have
$A^2_{n,m}(G)_S=A^2_{n,m}(\widehat{G})_{\widehat{S}}$.
If an arc which contains the base point $*$ of $G$ is labeled by $2$,
we consider a descending separating state $S$ of $G$ which has the
same arrows as $\widehat{S}$, but opposite labels. By Lemma
\ref{lem:essential-arrows},
$D^2_{n,m}(G)_S=A^2_{n,m}(\widehat{G})_{\widehat{S}}$.
We repeat this process, replacing ascending separating states with
descending and $A$ with $D$. Summing up by separating states, in
view of Corollary \ref{cor:enhanced-essential-arrows} we obtain the
following

\begin{thm}\label{thm:independence-basepoint2}
Let $G$ be a based Gauss diagram of an $m$-component
link $L$. Then $AD_{n,m}(G)$ is independent of the position
of the base point.
\end{thm}

Note that $AD_{n,m}(G)$ is invariant under $\Omega_1^F$ move. Now Theorems \ref{thm:invariance-2-3} and \ref{thm:independence-basepoint2} imply two important corollaries.

\begin{cor}\label{cor:invariance-framed}
Let $G$ be any Gauss diagram of an $m$-component framed
link $L$. Then $AD_n(L)=AD_{n,m}(G)$ is an invariant of an underlying
framed link, i.e does not depend on $G$.
\end{cor}

\begin{cor}\label{cor:invariance}
Let $G$ be any Gauss diagram of an $m$-component link $L$.
Then $I_n(L):=I_{n,m}(G)$
is an invariant of an underlying link $L$, i.e. is independent of a choice of $G$.
\end{cor}

\section{Properties of $I_n$}

In this section we establish the skein relation for $I_{n,m}$.
Then we identify $I_{n,m}$ with coefficients of a certain
polynomial, which is derived from the HOMFLYPT polynomial.

\subsection{Skein relation}

In this part we establish the skein relation for $AD_{n,m}(G)$. First we recall a notion of the Conway triple of links.

Let $L_+$, $L_-$ and $L_0$ be a triple of links with diagrams
which are identical except for a small fragment, where $L_+$ and
$L_-$ have  a positive and a negative  crossing respectively, and
$L_0$ has a smoothed crossing, see Figure \ref{fig:triple}. Such a triple of links is called a \textit{Conway triple}.

\begin{figure}[htb]
\centerline{\includegraphics[height=0.8in]{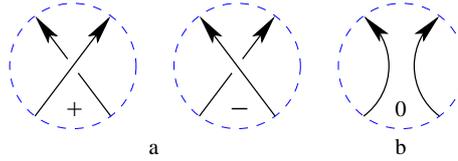}}
\caption{\label{fig:triple} Conway triple}
\end{figure}

\begin{thm}\label{thm:skein1}
Let $L_+$, $L_-$, $L_0$ be a Conway triple of links with the corresponding Conway triple $G_+$, $G_-$, $G_0$ of Gauss diagrams, see Figures \ref{fig:triple} and \ref{fig:Conway}. Denote the number
of circles of $G_\pm$ and $G_0$ by $m$ and $m_0$, respectively. Then
\begin{multline}\label{eq:A1}
AD_{n,m}(G_+)-AD_{n,m}(G_-)=\\
\begin{cases}
&AD_{n-1,m_0}(G_0)\ ,\quad \text{if $m_0=m-1$}\\
&AD_{n-1,m_0}(G_0)+\displaystyle\sum_{L'\subset L_0}
\sum_{i=0}^{n-1}c_i(L')c_{n-i-1}(L_0\sm L')\quad \text{if $m_0=m+1$},\\
\end{cases}
\end{multline}
where the summation is over all sublinks $L'$ of $L_0$ which contain exactly
one of the two new sublinks resulting from the smoothing.
\end{thm}

\begin{proof}
Denote the arrows of $G_+$ and $G_-$ appearing in Figure \ref{fig:Conway}
by $\a_+$ and $\a_-$, respectively. Let us look at labels of ascending separating states of $G_\pm$ and $G_0$ on
four arcs of the shown fragment. If labels of all four arcs are the same, we may
identify states of $G_\pm$ and $G_0$ with the same arrows and labels of arcs,
see Figure \ref{fig:skein-sep-states}a. Lemma \ref{lem:essential-arrows} and
Conway skein relation imply, that for every such state $S$
$$A^2_{n,m}(G_+)_S-A^2_{n,m}(G_-)_S=A^2_{n-1,m_0}(G_0)_S.$$
The descending case is treated similarly.

If labels on two arcs near the head of $\a_\pm$ coincide, but differ from
labels near the tail of $\a_\pm$, by Lemma \ref{lem:essential-arrows} we
have $A^2_{n,m}(G_+)_S-A^2_{n,m}(G_-)_S=0$ for any such state $S$
of $G_\pm$, and there is no corresponding state of $G_0$, see Figure
\ref{fig:skein-sep-states}b. The descending case is treated similarly.

\begin{figure}[htb]
\centerline{\includegraphics[width=4.7in]{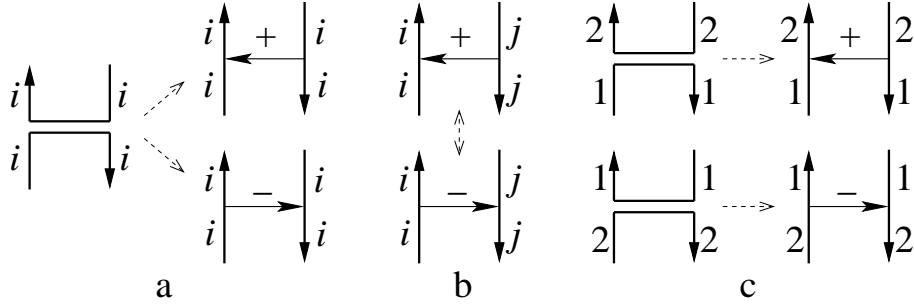}}
\caption{\label{fig:skein-sep-states} Correspondence of separating states
of $G_0$ and $G_\pm$.}
\end{figure}

There are two further cases when labels of two arcs near the head of
$\a_\pm$ are different. Such a state $S$ of $G_0$ corresponds either
to an ascending separating state $S\cup\a_+$ of $G_+$, or to an
ascending separating state $S\cup\a_-$ of $G_-$, see Figure
\ref{fig:skein-sep-states}c. By Lemma \ref{lem:essential-arrows} we have
$A^2_{n,m}(G_+)_{S\cup\a_+}=A^2_{n-1,m_0}(G_0)_S$ in the first case
and $A^2_{n,m}(G_-)_{S\cup\a_-}=-A^2_{n-1,m_0}(G_0)_S$ in the second case.

If $m_0=m-1$, there are no other ascending separating states of any of
the diagrams and, repeating this computation for descending separating
states, summing over states and using Corollary \ref{cor:enhanced-essential-arrows},
we obtain the first equality in~\eqref{eq:A1}.

If $m_0=m+1$, both ends of $\a_\pm$ are on the same circle of
$G_\pm$ and there is an additional contribution to $AD_{n,m}(G_\pm)$
of separating states $S=\{\a_\pm\}$ of $G_\pm$, which contain only the
arrow $\a_\pm$ (and some labeling of arcs)\footnote{These states have
no counterpart in $G_0$, since such a separating state of $G_0$ should
be empty and corresponding surface disconnected.}.
Such separating states correspond to labeling all circles of $G_0$ by 1, 2
so that the based circle is labeled by 1, and two new components of $G_0$
resulting from the smoothing have different labels. Denote by $L'$ the
sublink labeled by 1. The case of descending separating states $\{\a_\pm\}$
is similar. By Corollary \ref{cor:enhanced-essential-arrows}, the contribution
of these states to $AD_{n,m}(G_+)-AD_{n,m}(G_-)$ equals
$$\displaystyle\sum_{L'\subset L_0}\sum_{i=0}^{n-1}c_i(L')c_{n-i-1}(L_0\sm L')$$
and the proof follows.
\end{proof}

\subsection{Identification of the invariant $I_n$}

In this subsubsection we identify $I_n$ with certain derivatives
of the HOMFLYPT polynomial.

Let $P(L)$ be the HOMFLYPT polynomial of a link $L$.
We denote by $P'_a(L)$ the first derivative of $P(L)$ w.r.t. $a$.
Then $P'_a(L)|_{a=1}$ is a polynomial in the variable $z$. We
denote by $p_n(L)$ the coefficient of $z^n$ in $zP'_a(L)|_{a=1}$. Note that $p_n(L)$ a finite type invariant of degree $n$, see \cite{KM}, and hence by the Goussarov theorem it admits a Gauss diagram formula involving arrow diagrams with up to $n$ arrows. The precise formula is shown in the next theorem.

\begin{thm}\label{thm:identification}
Let $L$ be an $m$-component link. Then for every $n\geq 0$
\begin{equation}
I_n(L)=p_n(L)-C_n(L),
\end{equation}
where $C_n(L)$ is defined by
$$C_n(L):=\sum\limits_{L'\subset L}\sum_{i=0}^{n}c_i(L')c_{n-i}(L\sm L'),$$
and the summation is over all proper sublinks $L'$ of $L$.
\end{thm}

\begin{proof}
It is enough to show that $I_n(L)+C_n(L)$ and $p_n(L)$
satisfy the same skein relation and  take the same value on
unlinks with any number of components.

The skein relation for $zP'_a(L)|_{a=1}$ follows directly from
the skein relation for $P(L)$, see \eqref{eq:Homfly-skein}. Differentiating this skein relation
w.r.t. $a$, substituting $a=1$, and multiplying by $z$ we obtain
$$zP'_a(L_+)|_{a=1}-zP'_a(L_-)|_{a=1}+zP(L_+)|_{a=1}+zP(L_-)|_{a=1}=z^2P'_a(L_0)|_{a=1}.$$
Note that $P(L)|_{a=1}=\nabla(L)$ is the Conway polynomial
of $L$. Thus
$$zP'_a(L_+)|_{a=1}-zP'_a(L_-)|_{a=1}+z\nabla(L_+)+z\nabla(L_-)=z^2P'_a(L_0)|_{a=1}.$$
Taking the $n$-th coefficient, we get

\begin{equation}
p_n(L_+)-p_n(L_-)+c_{n-1}(L_+)+c_{n-1}(L_-)=
p_{n-1}(L_0).
\end{equation}

The skein relation for $C_n$ is obtained directly from the Conway skein relation.
It depends on the number $m_0$ of the components in $L_0$:

\begin{equation}\label{eq:C-n}
C_n(L_+)-C_n(L_-)=\begin{cases}
&C_{n-1}(L_0),\ \textrm{if $m_0=m-1$}\\
&2\displaystyle\sum_{L'\subset L_0}\sum_{i=0}^{n-1}c_i(L')c_{n-i-1}(L_0\sm L')\ \textrm{if $m_0=m+1$},\\
 \end{cases}
\end{equation}
where the summation is over all sublinks $L'$ of $L_0$ which
contain both new components appearing after the smoothing.
Now Theorem \ref{thm:skein1} and equality ~\eqref{eq:C-n} yield

\begin{equation*}
AD_{n,m}(G_+)-AD_{n,m}(G_-)+C_n(L_+)-C_n(L_-)=
AD_{n-1,m_0}(G_0)+C_{n-1}(L_0).
\end{equation*}

Deducting $w(G_0)c_{n-2}(L_0)$ from both sides of this equation
and noticing that $w(G_0)=w(G_+)-1=w(G_-)+1$ and
$c_{n-1}(L_+)-c_{n-1}(L_-)=c_{n-2}(L_0)$, so
$$w(G_0)c_{n-2}(L_0)=w(G_+)c_{n-1}(L_+)-w(G_+)c_{n-1}(L_+)-c_{n-1}(L_+)-c_{n-1}(L_-),$$
we obtain the desired skein relation for $I_n(L)+C_n(L)$:

\begin{align*}
(I_n(L_+)+C_n(L_+))-(I_n(L_-)+C_n(L_-))
+&c_{n-1}(L_+)+c_{n-1}(L_-)=\\ &I_{n-1}(L_0)+C_{n-1}(L_0).
\end{align*}

It remains to compare values of $I_n(L)+C_n(L)$ and $p_n(L)$ on an
$m$-component unlink $O_m$. From the definition of $I_n(L)$ we get
$I_n(O_m)=AD_{n,m}(O_m)=0$ for any $n$ and $m$. Also, the
equality $p_n(O_m)=C_n(O_m)$ holds for any $n$ and $m$, since
$p_0(O_2)=C_0(O_2)=2$ and $p_n(O_m)=C_n(O_m)=0$ otherwise.
This concludes the proof of the theorem.
\end{proof}

\begin{ex}\rm
Let $G$ be a Gauss diagram of a link $H_2$ shown in Figure \ref{fig:H2-Gauss}.

\begin{figure}[htb]
\centerline{\includegraphics[height=0.7in]{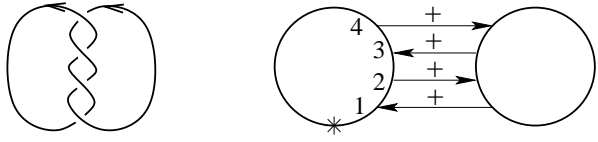}}
\caption{\label{fig:H2-Gauss} Link $H_2$ and its Gauss diagram.}
\end{figure}

Let us calculate $C_n(H_2)$ and $I_n(H_2)$. Both components of $H_2$
are trivial, so $C_0(H_2)=2$ and $C_n(H_2)=0$ for $n\ne 0$. The only
ascending states of $G$ are $\{1,2\}$, $\{1,4\}$, and $\{3,4\}$; the only
descending state of $G$ is $\{2,3\}$. Thus $AD_{2,2}(G)=4$. Note that
$c_1(H_2)=2$ and $c_n(H_2)=0$ for $n\ne 1$, thus $I_2(H_2)=4-4\cdot 2=-4$
and $I_n(H_2)=0$ for $n\ne 2$. Indeed, one may check that
$P(H_2)=a^{-3}z^{-1}-a^{-5}z^{-1}+a^{-3}z+a^{-1}z$, so
$zP'_a(H_2)|_{a=1}=2-4z^2$.
\end{ex}

\section{Last Remarks}
\subsection{The case of knots}

In this subsection we define for every $n\geq 0$ another two
invariants $I_{A,n}$ and $I_{D,n}$ of classical knots.

\begin{defn}\rm
Let $G$ be a based Gauss diagram of a knot $K$. We
go on the circle of $G$ starting from the base point until we
return to the base point. Denote by $w_A(G)$ (respectively
$w_D(G)$) sum of signs of all arrows of $G$ which we pass
first at the arrowhead (respectively arrowtail).
\end{defn}

\begin{thm} Let $G$ be any
based Gauss diagram of a knot $K$. Then for every $n\geq 0$ both
\begin{align*}
&I_{A,n}(G):=A^2_{n,1}(G)-w_A(G)\cdot c_{n-1}(K),\\
&I_{D,n}(G):=D^2_{n,1}(G)-w_D(G)\cdot c_{n-1}(K)
\end{align*}
are invariants of a knot $K$.
\end{thm}

These invariants will be denoted by $I_{A,n}(K)$ and $I_{D,n}(K)$ respectively.

\begin{proof} We will prove the invariance of $I_{A,n}(G)$;
the proof for $I_{D,n}(G)$ is the same.

A well-known fact in knot theory is that for classical knots, theories of
closed and long knots are equivalent. Thus it suffices to prove the
invariance of $I_{A,n}(G)$ under Reidemeister moves $\Omega_1$ --
$\Omega_3$ applied away from the base point. Note that both $w_A(G)$
and $A^2_{n,1}(G)$ are invariant under $\Omega_2$ and $\Omega_3$
(see Lemma \ref{thm:invariance-2-3}). It remains to prove the invariance
of $I_{A,n}(G)$ under $\Omega_1$. Let $G$ and $\tG$ be two based Gauss diagrams that differ by an
application of $\Omega_1$, so that $\tG$ has an additional isolated arrow
$\a$ either as in Figure \ref{fig:O1}a, or as in Figure \ref{fig:O1}b.

\begin{figure}[htb]
\begin{equation*}
\ris{-8}{-3}{80}{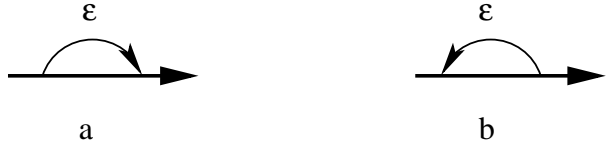}{-1.1}
\end{equation*}
\caption{\label{fig:O1} Two versions of the Reidemeister move $\Omega_1$.}
\end{figure}

In the first case, $A^2_{n,1}(\tG)=A^2_{n,1}(G)$ and $w_A(\tG)=w_A(G)$,
thus we have $I_{A,n}(G)=I_{A,n}(\tG)$. In the second case, by Corollary \ref{cor:Chmutov-Conway} we get
$$A^2_{n,1}(\tG)=A^2_{n,1}(G)+\eps c_{n-1}(K).$$
We also have $w_A(\tG)=w_A(G)+\eps$, and thus again $I_{A,n}(G)=I_{A,n}(\tG)$.
\end{proof}

Note that for every $G$ we have $w(G)=w_A(G)+w_D(G)$; also, for knots
one has $I_n(K)=p_n(K)$.
\begin{cor} For every knot $K$ we have $I_n(K)=I_{A,n}(K)+I_{D,n}(K)=p_n(K)$.
\end{cor}

\subsection{Irreducible arrow diagrams}

In this subsection we define a modification of link invariants considered in Section \ref{sec-1-comp}. This modification allows us to reduce significantly the number of diagrams in formulas for link invariants by using a special type of arrow diagrams -- so called irreducible diagrams.

\begin{defn}
\rm{An arrow diagram $A$ is called \textit{irreducible} if after the removal of any arrow in $A$ the remaining graph is connected. Otherwise, $A$ is called \textit{reducible}.}
\end{defn}

\begin{ex}
\rm{Irreducible diagrams in $\mathcal{A}_{3,2}$ are shown in Figure \ref{fig:A32}a, and reducible diagrams are shown in Figure \ref{fig:A32}b.}
\end{ex}
\begin{figure}[htb]
\centerline{\includegraphics[height=0.60in]{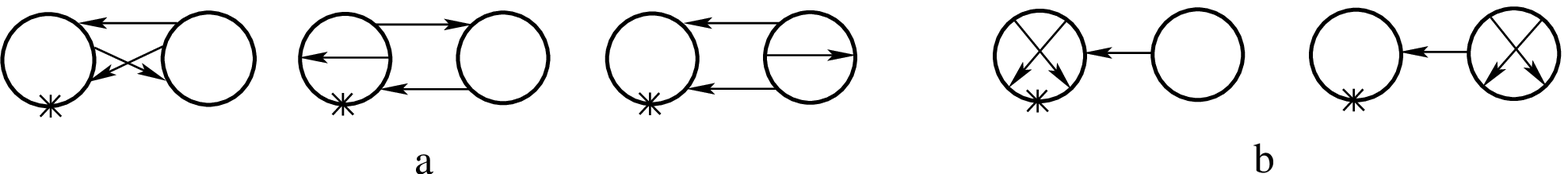}}
\caption{\label{fig:A32}}
\end{figure}

Denote by $\mathcal{AI}r_{n,m}\subset\mathcal{A}_{n,m}$ sets of all irreducible ascending diagrams with $m$ circles, $n$ arrows, and
one boundary component. Descending diagrams of the same types we will denote using $\mathcal{D}$ instead of $\mathcal{A}$. Let $G$ be a Gauss diagram of an $m$-component link $L$. Define $\mathcal{AI}r_{n,m}(G)$ and $\mathcal{DI}r_{n,m}(G)$ by
$$\mathcal{AI}r_{n,m}(G):=\sum_{A\in\mathcal{AI}r_{n,m}}\langle A,G\rangle\quad\quad
\mathcal{DI}r_{n,m}(G):=\sum_{A\in\mathcal{DI}r_{n,m}}\langle A,G\rangle.$$

\begin{thm}
Let $G$ be a Gauss diagram of an $m$-component link $L$. Then both $\mathcal{AI}r_{n,m}(G)$ and $\mathcal{DI}r_{n,m}(G)$ are invariants of an underlying link $L$. Moreover,
$$\mathcal{AI}r_{n,m}(G)=\mathcal{DI}r_{n,m}(G).$$
\end{thm}

\begin{proof}
For the simplicity we prove this theorem in case of two-component links, i.e. $m=2$. The proof for general $m$ is very similar and is left to the reader.

Let $G$ be any Gauss diagram of a two-component classical link $L=L_1\cup L_2$, and let $A\in\A_{n,2}$ be an arrow diagram with exactly one arrow between two different circles. The set of such arrow diagrams is denoted by $\overleftarrow{\A}_{n,2}$. We denote the set of descending diagrams of the same type by $\overrightarrow{\D}_{n,2}$.
We set
$$\overleftarrow{A}_{n,2}(G):=\sum\limits_{A\in \overleftarrow{\A}_{n,2}}\langle  A, G\rangle\quad\quad \overrightarrow{D}_{n,2}(G):=\sum\limits_{A\in \overleftarrow{\D}_{n,2}}\langle  A, G\rangle.$$
Now we prove that

\begin{equation}\label{eq:reducible-2-comp}
\overleftarrow{A}_{n,2}(G)=\lk(L_1,L_2)\sum\limits_{k=0}^{n-1} c_k(L_1)c_{n-k-1}(L_2)=\overrightarrow{D}_{n,2}(G).
\end{equation}

We start with the case of ascending diagrams. Let $G_1$ and $G_2$ be diagrams obtained from $G$ by erasing arrows between circles of $G$. We denote by $\mathfrak{A}(G)$ a set of arrows which are oriented from the non-based circle of $G$ to the based one. Without loss of generality suppose that a base point $*$ of $G$ lies on $G_1$.  We pick $\a\in \mathfrak{A}(G)$ and erase all other arrows in $\mathfrak{A}(G)$. The remaining diagram is denoted by $G_\a$. We place on $G_2$ a base point $*_\a$ at the tail of $\a$. Then
$$\sum\limits_{A\in \overleftarrow{\A}_{n,2}}\overleftarrow{A}_{n,2}(G_\a)=\s(\a)\sum\limits_{k=0}^{n-1} A_{k,2}(G_1)A_{n-k-1,2}(G_2)=\s(\a)\sum\limits_{k=0}^{n-1} c_k(L_1)c_{n-k-1}(L_2),$$
where the last equality is by Corollary \ref{cor:Chmutov-Conway}. It follows that
$$\overleftarrow{A}_{n,2}(G)=\sum\limits_{\a\in\mathfrak{A}(G)}\sum\limits_{A\in \overleftarrow{\A}_{n,2}}\overleftarrow{A}_{n,2}(G_\a)=\lk(L_1,L_2)\sum\limits_{k=0}^{n-1} c_k(L_1)c_{n-k-1}(L_2).$$
In case of descending diagrams, we denote by $\mathfrak{D}(G)$ a set of arrows which are oriented from the based circle of $G$ to the non-based one. For $\a\in\mathfrak{D}(G)$ we place a base point $*_\a$ at the head of $\a$. Now we proceed as in the former case and the proof of \eqref{eq:reducible-2-comp} follows. Note that by definition $\mathcal{AI}r_{n,2}=\A_{n,2}\setminus\overleftarrow{\A}_{n,2}$ and $\mathcal{DI}r_{n,2}=\D_{n,2}\setminus\overrightarrow{\D}_{n,2}$. Now the proof follows immediately from Corollary~\ref{cor:Chmutov-Conway}.
\end{proof}

\bibliographystyle{alpha}

\bigskip

Max-Planck-Institut f$\ddot{\textrm{u}}$r Mathematik, 53111 Bonn, Germany\\
\emph{E-mail address:} \verb"brandem@mpim-bonn.mpg.de"

\end{document}